\documentclass[10pt,reqno]{amsart}
   \usepackage[centertags]{amsmath}
   \usepackage{amsfonts}
   \usepackage{amsmath}
   \usepackage{amssymb}
   \usepackage{amsthm}
   \usepackage{newlfont}
   \usepackage[latin1]{inputenc}

\usepackage{multirow, bigdelim}

\usepackage{epsfig}
\usepackage{pst-all}
\usepackage{pstricks}
\usepackage{pgf}
\usepackage{mathrsfs}
\usepackage{hyperref}
\usepackage{bm}

\usepackage{graphicx}

\usepackage{bm}

\usepackage{mathtools}
\usepackage{ifmtarg}

\usepackage{color,soul}
\usepackage{tikz} 
\usetikzlibrary{shapes.misc}

\makeatletter

    \newcommand\contFrac{\@ifstar{\@contFracStar}{\@contFracNoStar}}

    \def\singleContFrac#1#2{%
        \begin{array}{@{}c@{}}%
            \multicolumn{1}{c|}{#1}%
            \\%
            \hline%
            \multicolumn{1}{|c}{#2}%
        \end{array}%
    }

    \def\@contFracNoStar#1{%
        \mathchoice{
            \@contFracNoStarDisplay@#1//\@nil%
        }{
            \@contFracNoStarInline@#1//\@nil%
        }{
            \@contFracNoStarInline@#1//\@nil%
        }{
            \@contFracNoStarInline@#1//\@nil%
        }%
    }

    \def\@contFracNoStarDisplay@#1//#2\@nil{%
        \@ifmtarg{#2}{%
            #1%
        }{%
            #1+\cfrac{1}{\@contFracNoStarDisplay@#2\@nil}%
        }%
    }

        \def\@contFracNoStarInline@#1//#2\@nil{%
            \@ifmtarg{#2}{%
                #1%
            }{%
                #1 \@@contFracNoStarInline@@#2\@nil%
            }%
        }
        \def\@@contFracNoStarInline@@#1//#2\@nil{%
            \@ifmtarg{#2}{%
                + \singleContFrac{1}{#1}%
            }{%
                + \singleContFrac{1}{#1} \@@contFracNoStarInline@@#2\@nil%
            }%
        }

    \def\@contFracStar#1{%
        \mathchoice{
            \@contFracStarDisplay@#1////\@nil%
        }{
            \@contFracStarInline@#1//\@nil%
        }{
            \@contFracStarInline@#1//\@nil%
        }{
            \@contFracStarInline@#1//\@nil%
        }%
    }

    \def\@contFracStarDisplay@#1//#2//#3\@nil{%
        \@ifmtarg{#2}{%
            #1%
        }{%
            #1 + \cfrac{#2}{\@contFracStarDisplay@#3\@nil}%
        }%
    }

        \def\@contFracStarInline@#1//#2\@nil{%
            \@ifmtarg{#2}{%
                #1%
            }{%
                #1 \@@contFracStarInline@@#2\@nil%
            }%
        }
        \def\@@contFracStarInline@@#1//#2//#3\@nil{%
            \@ifmtarg{#3}{%
                - \singleContFrac{#1}{#2}%
            }{%
                - \singleContFrac{#1}{#2} \@@contFracStarInline@@#3\@nil%
            }%
        }
\makeatother

\def\cFrac#1#2{%
\begin{array}{@{}c@{}}\multicolumn{1}{c|}{#1}\\%
\hline\multicolumn{1}{|c}{#2}\end{array}}

\allowdisplaybreaks[3]

\theoremstyle{plain}
\newtheorem{theorem}{Theorem}[section]

\newtheorem{proposition}[theorem]{Proposition}

\theoremstyle{definition}

\theoremstyle{remark}
\newtheorem{remark}[theorem]{Remark}
\newtheorem*{remark*}{Remark}

\numberwithin{equation}{section}

\newcommand\D{\displaystyle}

\newcommand\ZZ{{\mathbb Z}}

\newcommand\PP{{\mathbb P}}

\newcommand\p{\mbox{$\mathfrak{p}$}}

\newcommand*{\temp}{\multicolumn{0}{|c}{}}

\title[AR factorizations for birth-death chains on $\mathbb{Z}$]{Absorbing-reflecting factorizations for birth-death chains \\ on the integers and their Darboux transformations}

\author{Manuel D. de la Iglesia}
\address{Manuel D. de la Iglesia\\
Instituto de Matem\'aticas, Universidad Nacional Aut\'onoma de M\'exico, Circuito Exterior, C.U., 04510, Ciudad de M\'exico, M\'exico.}
\email{mdi29@im.unam.mx}

\author{Claudia Juarez}
\address{Claudia Juarez\\
Instituto de Investigaciones en Matem\'aticas Aplicadas y en Sistemas, Universidad Nacional Aut\'onoma de M\'exico, Circuito Escolar 3000, C.U., 04510, Ciudad de M\'exico, M\'exico.}
\email{ClaudiaJrz@ciencias.unam.mx}
\date{\today}
\thanks{
}

\thanks{This work was partially supported by PAPIIT-DGAPA-UNAM grant IN104219 (M\'exico) and CONACYT grant A1-S-16202 (M\'exico).}

\date{\today}



\begin{document}

\maketitle

\begin{abstract}
We consider a new way of factorizing the transition probability matrix of a discrete-time birth-death chain on the integers by means of an absorbing and a reflecting birth-death chain to the state 0 and viceversa. First we will consider reflecting-absorbing factorizations of birth-death chains on the integers. We give conditions on the two free parameters such that each of the factors is a stochastic matrix. By inverting the order of the factors (also known as a Darboux transformation) we get new families of ``almost'' birth-death chains on the integers with the only difference that we have new probabilities going from the state $1$ to the state $-1$ and viceversa. On the other hand an absorbing-reflecting factorization of birth-death chains on the integers is only possible if both factors are splitted into two separated birth-death chains at the state $0$. Therefore it makes more sense to consider absorbing-reflecting factorizations of ``almost'' birth-death chains with extra transitions between the states $1$ and $-1$ and with some conditions. This factorization is now unique and by inverting the order of the factors we get a birth-death chain on the integers. In both cases we identify the spectral matrices associated with the Darboux transformation, the first one being a Geronimus transformation and the second one a Christoffel transformation of the original spectral matrix. We apply our results to examples of chains with constant transition probabilities.

\end{abstract}

\section{Introduction}

This paper is a continuation of our work \cite{dIJ} about stochastic factorizations of the transition probability matrix $P$ of a discrete-time birth-death chain on the integers $\ZZ$ (doubly infinite tridiagonal matrix) and the relation between the spectral matrices after performing the so-called discrete Darboux transformation (inverting the order of the factors). In \cite{dIJ} we considered UL and LU stochastic factorizations of $P$. That means that each of the factors represents either a pure-birth or a pure-death chain. We consider here a new approach consisting of having one of the factors as a \emph{reflecting} birth-death chain on $\ZZ$ from the state $0$ and the other factor as an \emph{absorbing} birth-death chain on $\ZZ$ to the state $0$ (see \eqref{PRR} and \eqref{PAA} below). We will name them RA or AR factorizations, respectively. In the case of birth-death chains on the nonnegative integers $\ZZ_{\geq0}$ (see \cite{GdI3}) there is no difference between UL (LU) stochastic factorizations and RA (AR) factorizations, but when the state space is $\ZZ$ these factorizations represent different chains. The main motivation for these stochastic factorizations is to divide the probabilistic model into two different and simpler experiments, and combine them together to obtain a simpler description of the original probabilistic model (see applications to urn models in \cite{GdI3, GdI4}). 

The importance of having the spectral matrices is that it is easy to analyze the corresponding Markov chains in terms of polynomials which arise as a solution of the eigenvalue equation. For the case of birth-death chains on $\mathbb{Z}_{\geq0}$ this was first done in a series of papers by S. Karlin and J. McGregor in the 1950s (see \cite{KMc2, KMc3, KMc6}). Apart from an explicit expression of the $n$-step transition probabilities and the invariant measure, it is possible to study some other probabilistic properties using spectral methods such as recurrence, absorbing times, first return times or limit theorems. In the last section of \cite{KMc6} one can find the first attempt to perform the spectral analysis of a discrete-time birth-death chain on $\ZZ$ using orthogonal polynomials. We will recall this approach in Section \ref{sec3}. After that, apart from \cite{Ber}, there are not so many references concerning the spectral analysis of doubly infinite tridiagonal (or Jacobi) operators acting on $\ell_{\pi}^2(\ZZ)$. In \cite{Pru}, W.E. Pruitt studied the case of bilateral birth-and-death processes. An example of this approach can be found in the last section of \cite{ILMV}. For a more theoretical work about the spectral theory of doubly infinite tridiagonal operators see \cite{MR, DIW}.

In Section \ref{secRA} we analyze the conditions under which we can perform a stochastic RA factorization of the form $P=P_RP_A$. There are two important differences with our previous work in \cite{dIJ}. First, we will have now \emph{two free parameters}, while in the case of UL or LU factorizations there was only one. As in \cite{dIJ} (see also \cite{GdI3}) we will show that these two free parameters have to be bounded from below by certain continued fractions if we want to guarantee that the factors are still stochastic matrices. Second, after relabeling the states, $P$ will be equivalent to a semi-infinite $2\times2$ block tridiagonal matrix $\bm P$ (see \eqref{label} and \eqref{P2} below). The RA factorization of $P$ is not a UL factorization, but after relabeling, the corresponding matrix factorization of $\bm P$, denoted by $\bm P=\bm P_R\bm P_A$ (see \eqref{BPRA} below), \emph{will be a UL block matrix factorization}. However, in \cite{dIJ}, the corresponding matrix factorization of $\bm P$ did not preserve the UL block structure of the original factorization of $P$. This fact will allow us to use techniques from the area of matrix-valued orthogonal polynomials, something that we could not do in \cite{dIJ}. After that we perform a discrete Darboux transformation. The two-parameter family of new matrices $\widetilde{P}=P_AP_R$ are also stochastic but they are not strictly tridiagonal matrices since there will be new probability transitions between the states $1$ and $-1$. Nevertheless the block Darboux transformation $\widetilde{\bm P}=\bm P_A\bm P_R$ will preserve the block tridiagonal structure of $\bm P$. We will obtain a relation between the spectral matrices $\bm\Psi$ and $\widetilde{\bm\Psi}$ associated with $P$ and $\widetilde P$, respectively, even if $\widetilde P$ is not tridiagonal. This relation is given by 
$$
\widetilde{\bm\Psi}(x)=S_0\bm\Psi_{\bm U}(x)S_0^T,
$$ 
for certain constant matrix $S_0$ and $\bm\Psi_{\bm U}(x)$ is a \emph{Geronimus transformation} of the original spectral matrix $\bm\Psi(x)$  (see Theorem \ref{thmp} below). We apply our results to the simplest discrete-time birth-death chain on $\ZZ$ with constant transition probabilities (also known as a random walk).

In Section \ref{secAR} we perform the same analysis but considering a stochastic AR factorization of the form $P=\widetilde{P}_A\widetilde{P}_R$. The first thing we realize is that the case where $P$ is a birth-death chain on $\ZZ$ is not interesting since both factors $\widetilde{P}_A$ and $\widetilde{P}_R$ will be splited into two separated birth-death chains at the state $0$. Therefore it will be better to start with an ``almost'' birth-death chain similar to the one mentioned in the previous paragraph, i.e., with extra probability transitions between the states $1$ and $-1$. These transition probabilities must be related with the nearest transition probabilities between the $-1, 0$ and $1$ states (see \eqref{conddd} below). Although now $P$ is not tridiagonal, the equivalent block matrix $\bm P$ after relabeling will be a block tridiagonal matrix and the corresponding matrix factorization of $\bm P$, denoted by $\bm P=\widetilde{\bm P}_A\widetilde{\bm P}_R$, \emph{will be a LU block matrix factorization}. We then analyze under what conditions we get a stochastic AR factorization and show that we also need some bounds related with certain continued fractions. An important difference in this case is that the stochastic factorization, if possible, is \emph{unique} and there will be no extra free parameters. After that we consider the Darboux transformation $\widehat{P}=\widetilde{P}_R\widetilde{P}_A$ which it is now a tridiagonal matrix, i.e., a birth-death chain on $\ZZ$. Similarly the Darboux transformation $\widehat{\bm P}=\widetilde{\bm P}_R\widetilde{\bm P}_A$ will preserve the block tridiagonal structure of $\bm P$. Again we will obtain a relation between the spectral matrices $\bm\Psi$ and $\widehat{\bm\Psi}$ associated with $P$ and $\widehat P$, respectively. This relation is given by a \emph{Christoffel transformation} of $\bm\Psi(x)$ of the form 
$$
\widehat{\bm\Psi}(x)=x\widetilde{S}_0^{-1}\bm\Psi(x)\widetilde{S}_0^{-1},
$$
for certain constant matrix $\widetilde{S}_0$  (see Theorem \ref{thmp2} below). Finally we apply our results to an ``almost'' birth-death chain with constant transition probabilities. The main difficulty now is to compute the spectral matrix $\bm\Psi(x)$. However, using the block tridiagonal structure of $\bm P$, we will be able to use some results of the theory of matrix-valued orthogonal polynomials to obtain an explicit expression of the spectral matrix.

\section{Reflecting-absorbing factorization}\label{secRA}

Let $\{X_t : t=0,1,\ldots\}$ be an irreducible discrete-time birth-death chain on the integers $\ZZ$ with transition probability matrix given by
\begin{equation}\label{P1}
P=\left(
\begin{array}{ccc|cccc}
\ddots&\ddots&\ddots&&&\\
&c_{-1}&b_{-1}&a_{-1}&&&\\
\hline
&&c_0&b_0&a_0&&\\
&&&c_1&b_1&a_1&\\
&&&&\ddots&\ddots&\ddots
\end{array}
\right).
\end{equation}
Since the birth-death chain is irreducible, i.e., it is possible to get to any state from any state, then $0<a_n,c_n<1, n\in\ZZ$. Also, since $P$ is stochastic, it has nonnegative entries and 
$$
c_n+b_n+a_n=1,\quad n\in\ZZ.
$$
A diagram of the transitions between the states is given by
\begin{figure}[h]
\centering
\includegraphics[scale=0.25]{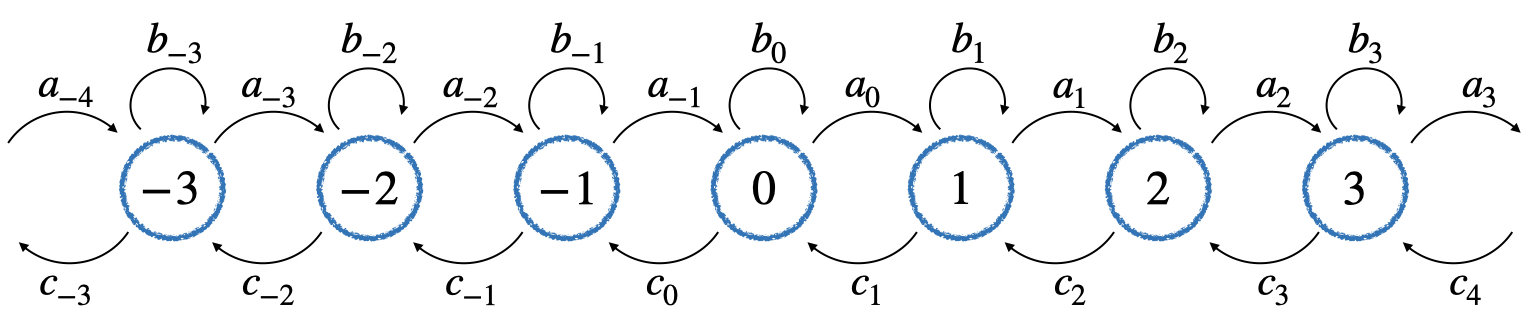}
\end{figure}
It is possible to relabel the states in such a way that all the information of $P$ is collected in a semi-infinite block tridiagonal matrix $\bm P $ with blocks of size $ 2 \times2 $. Indeed, after the new labeling
\begin{equation}\label{label}
\{0,1,2,\ldots\}\to\{0,2,4,\ldots\},\quad\mbox{and}\quad\{-1,-2,-3,\ldots\}\to\{1,3,5,\ldots\},
\end{equation}
we have that $P$ (doubly infinite tridiagonal) is equivalent to a semi-infinite $2\times2$ block tridiagonal matrix $\bm P$ of the form
\begin{equation}\label{P2}
\bm P=
\left(
\begin{array}{ccccccccccccc}
b_0&c_0&\temp &\hspace{-0.3cm} a_0&0&\temp & & & \\
a_{-1}&b_{-1}&\temp &\hspace{-0.3cm}0&c_{-1}&\temp & & & \\
\cline{1-8}
c_1&0&\temp &\hspace{-0.3cm}b_1&0&\temp &\hspace{-0.3cm} a_1&0&\temp& & \\
0&a_{-2}&\temp &\hspace{-0.3cm}0&b_{-2}&\temp &\hspace{-0.3cm}0&c_{-2}&\temp& & \\
\cline{1-11}
&&\temp &\hspace{-0.3cm}c_2&0&\temp &\hspace{-0.3cm}b_2&0&\temp&\hspace{-0.3cm} a_2 &0&\temp \\
&&\temp &\hspace{-0.3cm}0&a_{-3}&\temp &\hspace{-0.3cm}0&b_{-3}&\temp&\hspace{-0.3cm} 0 & c_{-3}&\temp \\
\cline{3-13}
&&&&&\temp &\hspace{-0.3cm}\ddots&&\temp &\hspace{-0.3cm} \ddots&&\temp&\hspace{-0.3cm} \ddots \\
\end{array}
\right)=\begin{pmatrix}
B_0&A_0&&&\\
C_1&B_1&A_1&&\\
&C_2&B_2&A_2&\\
&&\ddots&\ddots&\ddots
\end{pmatrix},
\end{equation}
where
\begin{equation}\label{ABC}
\begin{split}
B_0&=\begin{pmatrix} b_0 & c_0\\a_{-1} & b_{-1}\end{pmatrix},\quad B_n=\begin{pmatrix} b_n & 0\\ 0 & b_{-n-1}\end{pmatrix},\quad n\geq1,\\
A_n&=\begin{pmatrix} a_n & 0\\ 0 & c_{-n-1}\end{pmatrix},\quad n\geq0,\quad C_n=\begin{pmatrix} c_n & 0\\ 0 & a_{-n-1}\end{pmatrix},\quad n\geq1.
\end{split}
\end{equation}
The Markov chain generated by $\bm P $ takes values in the two-dimensional state space $\ZZ_{\geq0}\times \{1,2 \}. $ These type of processes are called discrete-time \emph {quasi-birth-and-death processes}. In general these processes allow transitions between all two-dimensional adjacent states (see \cite{LaR, Neu} for a general reference). The spectral analysis of these processes has been considered for instance in \cite{DRSZ, G2, G1, GdI2, dIR}. 

In \cite{dIJ} we considered stochastic UL and LU factorizations of $P$, which means that each of the factors can be viewed as a pure-birth or a pure-death chain, corresponding to the U or L matrix in the factorization, respectively. We consider here a different approach, given by a stochastic factorization of $P$ where the first factor is a \emph{reflecting} birth-death chain from the state 0 and the second factor is an \emph{absorbing} birth-death chain to the state 0. From now on we will call this factorization an \emph{RA factorization}. Therefore we are looking for a stochastic factorization of the transition probability matrix $P$ in \eqref{P1} of the form $P=P_RP_A$ where
\begin{equation}\label{PRR}
P_R=\left(
\begin{array}{cccc|cccc}
\ddots&\ddots&\ddots&&&\\
&x_{-2}&y_{-2}&0&&&&\\
&&x_{-1}&y_{-1}&0&&&\\
\hline
&&&\alpha&y_0&x_0&&\\
&&&&0&y_1&x_1&\\
&&&&&\ddots&\ddots&\ddots
\end{array}
\right),
\end{equation}
and
\begin{equation}\label{PAA}
P_A=\left(
\begin{array}{cccc|cccc}
\ddots&\ddots&\ddots&&&\\
&0&s_{-2}&r_{-2}&&&&\\
&&0&s_{-1}&r_{-1}&&&\\
\hline
&&&0&1&0&&\\
&&&&r_1&s_1&0&\\
&&&&&\ddots&\ddots&\ddots
\end{array}
\right).
\end{equation}
Observe that we have to add a new probability $\alpha$ in \eqref{PRR} in order to connect the reflecting birth-death chain from the state $0$ to the state $-1$. Also the state $0$ is an absorbing state in the chain \eqref{PAA}. Diagrams of the possible transitions between the states of both birth-death chains are given by
\begin{figure}[h]
\centering
\includegraphics[scale=0.4]{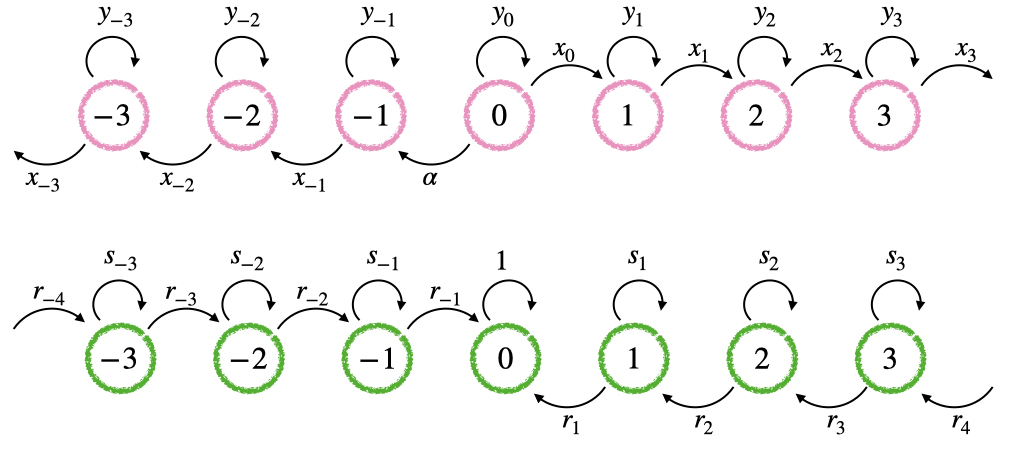}
\end{figure}

If we relabel the states as in \eqref{label}, then the matrix $P_R$ is equivalent to a semi-infinite $2\times2$ \emph{upper} block matrix $\bm P_R$ of the form
\begin{equation*}\label{PR}
\bm P_R=
\left(
\begin{array}{ccccccccccc}
y_0&\alpha&\temp &\hspace{-0.3cm} x_0&0&\temp & & & \\
0&y_{-1}&\temp &\hspace{-0.3cm}0&x_{-1}&\temp & & & \\
\cline{1-8}
&&\temp &\hspace{-0.3cm}y_1&0&\temp &\hspace{-0.3cm} x_1&0&\temp& & \\
&&\temp &\hspace{-0.3cm}0&y_{-2}&\temp &\hspace{-0.3cm}0&x_{-2}&\temp& & \\
\cline{3-11}
&&&&&\temp &\hspace{-0.3cm}\ddots&&\temp &\hspace{-0.3cm} \ddots& \\
\end{array}
\right),
\end{equation*}
while $P_A$ is equivalent to a semi-infinite $2\times2$ \emph{lower} block matrix $\bm P_A$ of the form
\begin{equation*}\label{PA}
\bm P_A=
\left(
\begin{array}{cccccccc}
1&0&\temp &\hspace{-0.3cm}&&    \\
r_{-1}&s_{-1}&\temp &\hspace{-0.3cm}&&    \\
\cline{1-5}
r_1&0&\temp &\hspace{-0.3cm}s_1&0&\temp &\hspace{-0.3cm}&  \\
0&r_{-2}&\temp &\hspace{-0.3cm}0&s_{-2}&\temp &\hspace{-0.3cm}&  \\
\cline{1-8}
&&\temp &\hspace{-0.3cm}\ddots&&\temp &\hspace{-0.3cm} \ddots& \\
\end{array}
\right).
\end{equation*}
If we call
\begin{equation}\label{XYSR}
\begin{split}
Y_0&=\begin{pmatrix}
y_0&\alpha\\
0&y_{-1}
\end{pmatrix},\quad
Y_n=\begin{pmatrix}
y_n&0\\
0&y_{-n-1}
\end{pmatrix}, \quad n \ge 1, \quad 
X_n=\begin{pmatrix}
x_n&0\\
0&x_{-n-1}
\end{pmatrix}, \quad n \ge 0,\\
S_0&=\begin{pmatrix}
1&0\\
r_{-1}&s_{-1}
\end{pmatrix},\quad
S_n=\begin{pmatrix}
s_n&0\\
0&s_{-n-1}
\end{pmatrix}, \quad n \ge 1, \quad 
R_n=\begin{pmatrix}
r_n&0\\
0&r_{-n-1}
\end{pmatrix}, \quad n \ge 0, 
\end{split}
\end{equation}
then we can write $\bm P_R$ and $\bm P_A$ as
\begin{equation}\label{BPRA}
\bm P_R=\left(
\begin{array}{ccccc}
Y_0&X_0&&\\
 &Y_1&X_1&\\
 & &Y_2&X_2\\
  && &\ddots&\ddots\\
\end{array}
\right), \quad 
\bm P_A=\left(
\begin{array}{cccc}
S_0&&&\\
 R_1&S_1&&\\
  &R_2&S_2&\\
   &&\ddots&\ddots\\
\end{array}
\right).
\end{equation}
The RA factorization $P=P_R P_A$, \emph{which is not a UL factorization}, is equivalent to the block matrix factorization $\bm P= \bm P_R \bm  P_A$, \emph{which is a UL block matrix factorization}. This is one of the main differences between this approach and the one we used in \cite{dIJ}, where we performed a UL factorization of $P$ while, after relabeling, the equivalent factorization for $\bm P$ did not have a UL block structure.

A direct computation from $\bm P= \bm P_R \bm  P_A$ shows that
\begin{equation}\label{ABC-XYSR}
\begin{split}
A_n&=X_nS_{n+1}, \quad n\ge 0,  \\
B_n&=Y_nS_n+X_nR_{n+1}, \quad n\ge0,\\
C_n&=Y_nR_n, \quad n\ge 1,
\end{split}
\end{equation}
or equivalently, using  $P=P_R P_A$, we obtain
\begin{align}
&a_n=x_ns_{n+1}, \quad n\ge 0,\quad a_{-n}=y_{-n}r_{-n},\quad n\ge 1,\label{aa}\\
\nonumber &b_n=y_ns_n+x_nr_{n+1}, \quad n\ge 1, \quad b_0=y_0+x_0r_1+\alpha r_{-1},\quad b_{-n}=y_{-n}s_{-n}+x_{-n}r_{-n-1},\quad n\ge 1,\\
&c_n=y_nr_n, \quad n\ge 1, \quad c_0=\alpha s_{-1}, \quad c_{-n}=x_{-n}s_{-n-1},\quad n\ge 1\label{cc}.
\end{align}
Now we will see under what conditions we have that $P_R$ and $P_A$ are also stochastic matrices, i.e.,  all entries are nonnegative and
\begin{equation}\label{x0s}
\alpha+x_0+y_0=1,\quad x_n+y_n=1, \quad s_n+r_n=1, \quad n\in\ZZ \setminus \{0\}.
\end{equation}
If we fix $\alpha$ then we can compute $s_{-1}$, $r_{-1}$, $y_{-1}$, $x_{-1}$, $s_{-2},\ldots$ recursively using \eqref{aa}, \eqref{x0s} and \eqref{cc}. On the other hand, for positive values of the indices, we need to fix a second parameter, say $x_0$, in order to obtain recursively $s_1$, $r_1$, $y_1$, $x_1$, $s_2,\ldots$ using again \eqref{aa}, \eqref{x0s} and \eqref{cc}. Using the same arguments as in \cite{dIJ} we will see under what conditions on the free parameters $\alpha$ and $x_0$ we have that both matrices $P_R$ and $P_A$ are also stochastic matrices. For that consider $H$ and $H'$ the continued fractions 
\begin{equation}\label{ccff}
H=\cFrac{a_0}{1}-\cFrac{c_1}{1}-\cFrac{a_1}{1}-\cFrac{c_2}{1}-\cdots,\quad H'=\cFrac{c_0}{1}-\cFrac{a_{-1}}{1}-\cFrac{c_{-1}}{1}-\cFrac{a_{-2}}{1}-\cdots
\end{equation}
For each continued fraction, consider the corresponding sequence of convergents
$(h_{n})_{n\geq0}$ and $(h_{-n}')_{n\geq0}$, given by
\begin{equation}\label{hhh}
h_{n}=\frac{A_n}{B_n},\quad h_{-n}'=\frac{A_{-n}'}{B_{-n}'}.
\end{equation}
We refer to \cite{GdI2, dIJ} to find more information about the notation and definitions on continued fractions.
\begin{proposition}\label{propalfx0}
Let $H$ and $H'$ be the continued fractions defined by \eqref{ccff} and the corresponding convergents $h_n$ and $h_{-n}$ defined by \eqref{hhh}. Assume that 
\begin{equation*}\label{AnBn}
0<A_n<B_n,\quad\mbox{and}\quad0<A_{-n}'<B_{-n}',\quad n\geq1.
\end{equation*}
Then both $H$ and $H'$ are convergent. Moreover, let $P=P_RP_A$ and assume that $H+H'\leq1$. Then, both $P_R$ and $P_A$ are stochastic matrices if and only if we choose $\alpha$ and $x_0$ in the following ranges
\begin{equation*}\label{yy0r}
\alpha \ge H',\quad x_0 \ge H.
\end{equation*}
\end{proposition}
\begin{proof}
The proof follows the same lines as the proof of Theorem 2.1 of \cite{dIJ} and will be omitted (see also Theorem 2.1 of \cite{GdI3}).
\end{proof}


\subsection{Stochastic Darboux transformation and the associated spectral matrix}\label{sec3}

We have shown above under what conditions a doubly stochastic matrix $P$ like in \eqref{P1} can be decomposed as an RA factorization, where both factors are still stochastic matrices. This factorization comes with \emph{two free parameters}. Once we have this factorization it is possible to perform what is called as a \emph{discrete Darboux transformation}, consisting of inverting the order of the factors. The Darboux transformation has a long history but probably the first reference of a discrete Darboux transformation like we study here appeared in \cite{MS} in connection with the Toda lattice.

If $P=P_RP_A$ as in \eqref{PRR} and \eqref{PAA}, then, by inverting the order of the factors, we obtain another stochastic matrix of the form $\widetilde P=P_AP_R$, since the multiplication of two stochastic matrices is again a stochastic matrix. This matrix $\widetilde P$ \emph{is not tridiagonal} but \emph{pentadiagonal}, so we do not obtain again a birth-death chain on $\ZZ$, as in the case of UL (or LU) factorization in \cite{dIJ}. However we obtain a two-parameter family of new Markov chains on $\ZZ$, which we denote by $\{\widetilde X_t: t=0,1,\ldots\}$, which is ``almost'' a birth-death chain. The only difference is that we have new transition probabilities between the state $1$ to the state $-1$ and viceversa. The tridiagonal coefficients of $\widetilde P$ are given by
\begin{align}
\nonumber\tilde{a}_n&=s_nx_{n}, \quad\tilde{a}_{-n-1}=y_{-n}r_{-n-1},\quad n\geq0,\\
\label{DTabc}\tilde b_0&=y_0,\quad \tilde b_{-1}=r_{-1}\alpha+s_{-1}y_{-1},\quad \tilde{b}_n=r_nx_{n-1}+s_ny_n,\quad\tilde{b}_{-n}=r_{-n}x_{-n+1}+s_{-n}y_{-n},\quad n\geq1,\\
\nonumber\tilde{c}_0&=\alpha,\quad \tilde{c}_n=r_ny_{n-1},\quad\tilde{c}_{-n}=s_{-n}x_{-n},\quad n\geq1,
\end{align}
while the probability transitions between  the state $1$ and $-1$ are given by
\begin{equation}\label{DTdd}
\tilde d_{1}=\PP(\widetilde X_1=-1 | \widetilde X_0=1)=r_1\alpha,\quad \tilde d_{-1}=\PP(\widetilde X_1=1 | \widetilde X_0=-1)=r_{-1}x_0.
\end{equation}
A diagram of the transitions between the states of this new family of Markov chains is given by

\begin{figure}[h]
\centering
\includegraphics[scale=0.4]{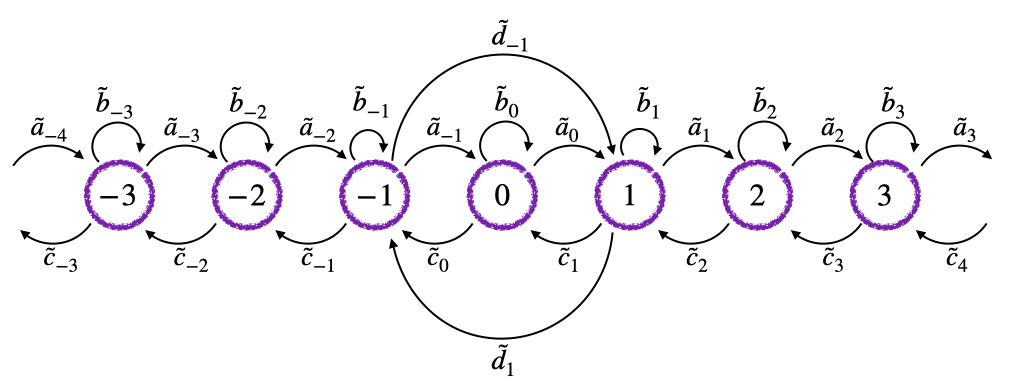}
\end{figure}

Another way to obtain these coefficients is by considering the block matrix factorization $\bm P= \bm P_R \bm  P_A$, where $\bm P$ is given by \eqref{P2} (see also \eqref{ABC}) and $\bm P_R$ and $\bm P_A$ are given by \eqref{BPRA} (see also \eqref{XYSR}). The Darboux transformation will be given by $\bm{\widetilde P}=\bm P_A \bm  P_R$, i.e.,
\begin{equation}\label{tildep}
\bm{\widetilde P}=\begin{pmatrix}
\widetilde B_0 & \widetilde A_0 & & \\
\widetilde  C_1& \widetilde B_1& \widetilde A_1& \\
& \widetilde C_2& \widetilde B_2 & \widetilde A_2\\
& &\ddots &\ddots&\ddots \\
\end{pmatrix}=\begin{pmatrix}
S_0 & & & \\
R_1& S_1&& \\
& R_2 & S_2 &\\
& & \ddots &\ddots \\
\end{pmatrix}
\begin{pmatrix}
Y_0 &X_0 & & \\
& Y_1&X_1& \\
& & Y_2 & X_2\\
& & &\ddots&\ddots \\
\end{pmatrix}.
\end{equation}
A direct computation shows
\begin{equation}\label{DTABC1}
\begin{split}
\widetilde A_n&= S_nX_n, \quad n\ge 0, \\
\widetilde B_0&=S_0 Y_0, \quad \widetilde B_n=R_nX_{n-1}+S_nY_n, \quad n\ge 1,\\
\widetilde C_n&=R_nY_{n-1}, \quad n\ge 0. \end{split}
\end{equation}
Using the notation in \eqref{ABC} we can obtain again the probabilities \eqref{DTabc} and \eqref{DTdd}. Since $\bm P= \bm P_R \bm  P_A$ is a block UL matrix factorization then $\bm{\widetilde P}=\bm P_A \bm  P_R$ will be also a block tridiagonal matrix, and will be again a family of discrete-time quasi-birth-and-death processes.

Now let us focus on the following question: given the spectrum of the doubly infinite matrix $P$, how can we compute the spectrum of the Darboux transformation $\widetilde{P}$ of the RA factorization? We will see that this will be related with what is called a \emph{Geronimus transformation}. Before that let us introduce some notation to study the spectral measures associated with the original birth-death chain $P$ on $\ZZ$.

We will follow the last section of \cite{KMc6} (see also \cite{dIJ}). For $P$ like in \eqref{P1} consider the eigenvalue equation $ x q^\alpha(x) = P q^\alpha(x) $ where $q^\alpha(x)=(\cdots, Q_{-1}^\alpha(x),Q_0^\alpha(x),Q_1^\alpha(x),\cdots)^T$, $ \alpha = 1, 2$. For each $ x $ real or complex there exist two polynomial families of linearly independent solutions $ Q_n ^ {\alpha} (x)$, $\alpha = 1, 2, n \in \ZZ , $ depending on the initial values at $n = 0 $ and $ n = -1 $. These polynomials are given by
\begin{align}
\nonumber Q_0^1(x)&=1,\quad Q_{0}^2(x)=0,\\
\label{TTRRZ}Q_{-1}^1(x)&=0,\quad Q_{-1}^2(x)=1,\\
\nonumber xQ_n^{\alpha}(x)&=a_nQ_{n+1}^{\alpha}(x)+b_nQ_n^{\alpha}(x)+c_nQ_{n-1}^{\alpha}(x),\quad n\in\ZZ,\quad\alpha=1,2.
\end{align}
Observe that 
\begin{equation*}\label{degqs}
\begin{split}
\deg(Q_n^1)&=n,\quad n\geq0,\hspace{.95cm}\deg(Q_n^2)=n-1,\quad n\geq1,\\
\deg(Q_{-n-1}^1)&=n-1,\quad n\geq1,\quad\deg(Q_{-n-1}^2)=n,\quad n\geq0.
\end{split}
\end{equation*}
Let us define the \emph{potential coefficients} as
\begin{equation}\label{potcoeff}
\pi_0=1,\quad \pi_n=\frac{a_0a_1\cdots a_{n-1}}{c_1c_2\cdots c_n},\quad \pi_{-n}=\frac{c_0c_{-1}\cdots c_{-n+1}}{a_{-1}a_{-2}\cdots a_{-n}},\quad n\geq1.
\end{equation}
These coefficients are defined as the solutions of the \emph{symmetry equations} $P_{ij}\pi_i=P_{ji}\pi_j$ normalized by the condition $\pi_0=1$. In particular we have that $\pi P=\pi$, i.e., $\pi=(\pi_n)_{n\in\ZZ}$ is an invariant vector of $P$. As a consequence of these symmetry equations, the matrix $P$ gives rise to a self-adjoint operator of norm $\leq1$ in the Hilbert space $\ell^2_{\pi}(\ZZ)$, which we will denote by $P$, abusing the notation. Applying the spectral theorem we obtain three unique measures $\psi_{\alpha\beta}, \alpha,\beta=1,2$ ($\psi_{12}=\psi_{21}$), which are supported on the interval $[-1,1]$ such that we have the following orthogonality relation
\begin{equation}\label{ortoZ}
\sum_{\alpha,\beta=1}^{2}\int_{-1}^{1}Q_i^{\alpha}(x)Q_j^{\beta}(x)d\psi_{\alpha\beta}(x)=\frac{\delta_{i,j}}{\pi_j},\quad i,j\in\ZZ.
\end{equation}
For more details see \cite{dIJ, KMc6, MR}. The measures $\psi_{11}$ and $\psi_{22}$ are positive (in fact $\psi_{11}$ is a probability measure but $\psi_{22}$ is not, since $\int_{-1}^1d\psi_{22}(x)=1/\pi_{-1}$). The measure $\psi_{12}$ is a signed measure satisfying $0=\int_{-1}^1d\psi_{12}(x)$. For simplicity let us assume that the three measures are continuously differentiable with respect to the Lebesgue measure, i.e., $d\psi_{\alpha\beta}(x)=\psi_{\alpha\beta}(x)dx$, $\alpha, \beta= 1, 2$, abusing the notation. These 3 measures can be written in matrix form as the $2\times2$ matrix
\begin{equation}\label{2spmt}
\bm\Psi(x)=\begin{pmatrix} \psi_{11}(x) & \psi_{12}(x)\\ \psi_{12}(x) & \psi_{22}(x) \end{pmatrix},
\end{equation}
so that the orthogonality relation \eqref{ortoZ} can be written in matrix form as
\begin{equation*}\label{ortoZ2}
\int_{-1}^{1}\left(Q_i^1(x),Q_i^2(x)\right)\bm\Psi(x)\begin{pmatrix}Q_j^1(x)\\Q_j^2(x)\end{pmatrix}dx=\frac{\delta_{i,j}}{\pi_j},\quad i,j\in\ZZ.
\end{equation*}
The matrix $\bm\Psi(x)$ in \eqref{2spmt} is called the \emph{spectral matrix} associated with $P$. The orthogonality conditions are valid for any indexes $ i, j \in \ZZ$.  With this information we can compute the $n$-step transition probabilities of the random walk $\{X_t : t=0,1,\ldots\}$, given by the so-called \emph{Karlin-McGregor integral representation formula} (see \cite{KMc6})
\begin{equation*}\label{KmcG1}
P_{ij}^{(n)}\doteq\mathbb{P}(X_n=j \; | X_0=i)=\pi_j\int_{-1}^{1}x^n\left(Q_i^1(x),Q_i^2(x)\right)\bm\Psi(x)\begin{pmatrix}Q_j^1(x)\\Q_j^2(x)\end{pmatrix}dx,\quad i,j\in\ZZ.
\end{equation*}

If we define the matrix-valued polynomials
\begin{equation}\label{2QMM}
\bm Q_n(x)=\begin{pmatrix} Q_n^1(x) & Q_n^2(x) \\ Q_{-n-1}^1(x) & Q_{-n-1}^2(x)\end{pmatrix},\quad n\geq0,
\end{equation}
then we have
\begin{equation}\label{ttrrq}
\begin{split}
x\bm Q_0(x)&=A_0\bm Q_{1}(x)+B_0\bm Q_0(x),\quad \bm Q_0(x)=I_{2\times2},\\
x\bm Q_n(x)&=A_n\bm Q_{n+1}(x)+B_n\bm Q_n(x)+C_n\bm Q_{n-1}(x),\quad n\geq1,
\end{split} 
\end{equation}
where $I_{2\times2}$ denotes the $2\times2$ identity matrix and $(A_n)_{n\geq0}$, $(B_n)_{n\geq0}$ and $(C_n)_{n\geq1}$ are given by \eqref{ABC}. If we denote $\bm{Q}=(\bm{Q}_0^T,\bm{Q}_1^T,\cdots)^T$ then we have $x\bm{Q}=\bm{P}\bm{Q}$, where $\bm P$ is given by \eqref{P2}. The matrix orthogonality is defined in terms of the (matrix-valued) inner product
\begin{equation*}\label{ortoq}
\int_{-1}^{1}\bm Q_n(x)\bm\Psi(x)\bm Q_m^T(x)dx=\Pi_n^{-1}\delta_{nm},
\end{equation*}
where $A^T$ is the transpose of a matrix $A$ and 
\begin{equation}\label{potcoeffm}
\Pi_n=\begin{pmatrix} \pi_n & 0 \\ 0 & \pi_{-n-1}\end{pmatrix},\quad n\geq0,
\end{equation}
where $\pi=(\pi_n)_{n\in\ZZ}$ are given by \eqref{potcoeff}. An alternative way of writing $\Pi_n$, which solves the symmetry equations for $\bm P$,  is given by (see \cite{dI1})
\begin{equation*}\label{pimatrix}
\Pi_n=(C_1^T\cdots C_n^T)^{-1}\Pi_0 A_0 \cdots A_{n-1}, \quad n\ge 1.
\end{equation*}
Therefore we have (see \cite{DRSZ, G2}) the Karlin-McGregor integral representation formula where the $2\times2$ block entry $(i,j)$ is given by
\begin{equation*}
\bm P_{ij}^{(n)}=\left(\int_{-1}^{1}x^n\bm Q_i(x)\bm\Psi(x)\bm Q_j^T(x)dx\right)\Pi_j,\quad i,j\in\ZZ_{\geq0}.
\end{equation*}

\begin{remark}
In several references (see for instance \cite{DIW, MR, Pru}) the two families of polynomials $(Q_n^\alpha)_{n\in\ZZ},\alpha=1,2,$ generated by the three-term recurrence relation \eqref{TTRRZ} are denoted in a different way and have different initial conditions. If we denote by $(P_n^\beta)_{n\in\ZZ}$ the families defined in those papers, then the initial conditions are given by $P_n^\beta=\delta_{n \beta}$ for $n,\beta=0,1$. We prefer to use our notation (and follow \cite{KMc6} or \cite{Ber}) since in this case, and after relabeling, the sequence $(\bm Q_n)_{n\geq0}$ in \eqref{2QMM} is a proper family of matrix-valued polynomials, i.e., $\deg\bm Q_n=n$ and have nonsingular leading coefficient, something that will not occur if we perform the same labeling to the family $(P_n^\beta)_{n\in\ZZ}$. In this way we make a direct connection between the theory of orthogonal polynomials on the integers and the theory of $2\times2$ matrix-valued orthogonal polynomials.

\end{remark}

\medskip

Consider now the matrix-valued polynomials
\begin{equation} \label{relacionqhat-q}
\begin{split}
\bm{U}_0(x)&= S_0 \bm Q_0(x)=S_0, \\
\bm{U}_n(x)&=R_n \bm Q_{n-1}(x)+S_n \bm Q_n(x), \quad n \ge 1,
\end{split}
\end{equation}
where $(S_n)_{n\geq0}$ and $(R_n)_{n\geq1}$ are defined by \eqref{XYSR}. If we denote $\bm{U}=(\bm{U}_0^T,\bm{U}_1^T,\cdots)^T$, then we have that $\bm{U}=\bm{P}_A\bm{Q}$, where $\bm{P}_A$ is given by \eqref{BPRA}. Therefore, from the RA factorization of $\bm P$, we get $\bm P_R\bm{U}= \bm P_R \bm P_A \bm Q= \bm P\bm Q=x\bm Q,$
or in other words
\begin{equation}\label{relacionq-qhat}
x\bm Q_n(x)=Y_n \bm{U}_n(x)+X_n\bm{U}_{n+1}(x), \quad n\ge 0.
\end{equation}
Finally, using the Darboux transformation \eqref{tildep}, we have
\begin{equation*}\label{mfsalf}
\widetilde{\bm{P}}\bm{U} = \bm P_A\bm P_R \bm P_A\bm{Q} = x\bm P_A\bm{Q} = x\bm U.
\end{equation*}
Evaluating at $x=0$ in \eqref{relacionq-qhat} we can easily get
\begin{equation}\label{qhatin0}
\bm{U}_n(0)=(-1)^{n}X_{n-1}^{-1}Y_{n-1}\cdots X_0^{-1}Y_0S_0.
\end{equation}
As a consequence we also have
\begin{equation*}\label{hatqnqn-k}
\bm{U}_{n+1}(0)\bm{U}_{n-k}^{-1}(0)=(-1)^{k+1}X_{n}^{-1}Y_n \cdots X_{n-k}^{-1}Y_{n-k}, \quad k=0,1,\ldots, n.
\end{equation*}
We can also solve \eqref{relacionq-qhat} recursively in which case we have, using the previous notation,
\begin{equation}\label{qhatsuma}
\bm{U}_n(x)=\bm{U}_n(0)\left[ I+x\sum_{k=0}^{n-1} \bm{U}_{k+1}^{-1}(0) X_k^{-1}\bm Q_k(x)\right].
\end{equation}
From \eqref{relacionqhat-q} we have that $\mbox{deg}(\bm{U}_n(x))=n, n\geq0,$ but $\bm{U}_0(x)=S_0\neq I_{2\times 2}$. We will be interested in a new family of matrix-valued polynomials $(\widetilde{\bm{Q}}_n)_{n\geq0}$ where $\widetilde{\bm{Q}}_0=I_{2\times 2}$. Since $S_0$ is a constant matrix and has an inverse, this new family can be defined as
\begin{equation}\label{qtilde-qhat}
\bm{\widetilde Q}_n(x)=\bm{U}_n(x)S_0^{-1}.
\end{equation}
Finally, if we define $(\widetilde \Pi_n)_{n\geq0}$ as the solution of the symmetry equations for $\bm{\widetilde P}$ in \eqref{tildep}, given by 
\begin{equation*}
\widetilde \Pi_n=(\widetilde C_1^T\cdots \widetilde C_n^T)^{-1} \widetilde  \Pi_0 \widetilde  A_0 \cdots \widetilde  A_{n-1}, \quad n\ge 1,
\end{equation*}
then, using \eqref{ABC-XYSR}, \eqref{DTABC1} and the previous equation, we get
\begin{equation}\label{coefpot1}
\widetilde \Pi_n=Y_n^{T} \Pi_n S_n^{-1}, \quad n\ge 0.
\end{equation}
Computing $\widetilde\Pi_0$ using \eqref{aa} and \eqref{cc}, we get
\begin{equation*}
\widetilde \Pi_0=\begin{pmatrix}  y_0&0\\ 0& \alpha/r_{-1} \end{pmatrix}.
\end{equation*}
Therefore $(\widetilde \Pi_n)_{n\geq0}$ are always \emph{diagonal matrices}. We are ready to prove the main result of this section.
\begin{theorem}\label{thmp}
Let $\{X_t: t=0, 1, \dots\}$ be the birth-death chain on $\ZZ$ with transition probability matrix $P$ given by \eqref{P1} and $\{\widetilde X_t : t=0, 1, \dots\}$ be the Markov chain generated by the Darboux transformation of $P=P_RP_A$ with transition probabilities given by \eqref{DTabc} and \eqref{DTdd}. Assume that $\bm M_{-1}=\D\int_{-1}^{1} x^{-1}\bm\Psi(x)dx$ is well-defined (entry by entry), where $\bm\Psi(x)$ is the original spectral matrix \eqref{2spmt}. Then the matrix-valued polynomials $(\widetilde{\bm Q}_n)_{n\geq0}$ defined by \eqref{qtilde-qhat} are orthogonal with respect to the following spectral matrix 
\begin{equation}\label{psitilde}
\bm{\widetilde \Psi}(x)=S_0\bm{\Psi_U}(x) S_0^T,
\end{equation}
where the constant matrix $S_0$ is defined by \eqref{XYSR} and
\begin{equation}\label{psihat}
\bm{ \Psi_U}(x)=\frac{\bm \Psi(x)}{x}+\left[ \frac{1}{y_0}\begin{pmatrix} 1&-r_{-1}/s_{-1}\\ -r_{-1}/s_{-1}& (r_{-1}/s_{-1})^2+y_0r_{-1}/\alpha s^2_{-1}\end{pmatrix}-\bm M_{-1}\right]\delta_0(x).
\end{equation}
Moreover, we have
\begin{equation*}\label{qusnorms}
\int_{-1}^{1} \widetilde{\bm Q}_n (x) \widetilde{\bm\Psi }(x) \widetilde{\bm Q}_m^T(x) dx=\widetilde{\Pi}_n^{-1}\delta_{n,m},
\end{equation*}
where $(\widetilde{\Pi}_n)_{n\geq0}$ are defined by \eqref{coefpot1}.
\end{theorem}

\begin{proof}
For $n\ge 1$ and $j=1, \dots, n-1,$ we have 
\begin{equation*}
\begin{split}
\int_{-1}^{1} \bm{\widetilde Q}_n(x) \bm {\widetilde \Psi}(x) x^j dx &=\int_{-1}^1 \bm{U}_n(x) \bm{ \Psi_U}(x) x^j S_0^Tdx= \int_{-1}^1 [R_n \bm{Q}_{n-1}(x)+S_n\bm Q_n(x)] \bm{\Psi}(x) x^{j-1}S_0^Tdx\\
&= \int_{-1}^1 R_n \bm{Q}_{n-1}(x)\bm{ \Psi}(x) x^{j-1} S_0^Tdx+ \int_{-1}^1 S_n\bm Q_n(x)\bm{ \Psi}(x) x^{j-1}S_0^Tdx=\bm 0_{2\times2},
\end{split}
\end{equation*}
where for the first equality we have used \eqref{qtilde-qhat} and \eqref{psitilde}, for the second equality we have used \eqref{relacionqhat-q} and \eqref{psihat}, and finally we have used the orthogonality of the family $(\bm Q_n)_{n\geq0}$. Now, for $n\geq1$ we have, using \eqref{qhatsuma}, that
\begin{equation*}
\begin{split}
\int_{-1}^{1} \bm{\widetilde Q}_n(x) \bm {\widetilde \Psi}(x) dx &=\int_{-1}^{1} \bm{U}_n(x) \bm{\Psi_U}(x)S_0^Tdx=\int_{-1}^{1} \bm{U}_n(0)\left[ I+x\sum_{k=0}^{n-1} \bm{U}_{k+1}^{-1}(0) X_k^{-1}\bm Q_k(x)\right] \bm{ \Psi_U}(x)S_0^Tdx\\
&= \bm{U}_n(0) \left[ \int_{-1}^{1} \bm{\Psi_U}(x)dx +\sum_{k=0}^{n-1} \bm{U}_{k+1}^{-1}(0) X_k^{-1}  \int_{-1}^{1}  \bm Q_k(x) \bm \Psi(x)dx  \right] S_0^T.
\end{split}
\end{equation*}
The second part of the previous sum vanishes for $k=1,\ldots, n-1$. Therefore the only nonzero term is for $k=0$, i.e., $\bm{U}_1^{-1}(0)X_0^{-1}\int_{-1}^1 \bm \Psi(x)dx=\bm{U}_1^{-1}(0)X_0^{-1}\Pi_0^{-1}$. A direct computation using \eqref{qhatin0}, \eqref{XYSR}, \eqref{potcoeffm} and \eqref{potcoeff} shows that 
\begin{equation*}
\bm{U}_1^{-1}(0)X_0^{-1}\Pi_0^{-1}=-S_0^{-1}Y_0^{-1}\Pi_0^{-1}=-\frac{1}{y_0}\begin{pmatrix} 1&-r_{-1}/s_{-1}\\ -r_{-1}/s_{-1}& (r_{-1}/s_{-1})^2+y_0r_{-1}/\alpha s^2_{-1}\end{pmatrix}.
\end{equation*}
From the definition of $\bm{ \Psi_U}(x)$ in \eqref{psihat} we obtain that $\int_{-1}^{1} \bm{\widetilde Q}_n(x) \bm {\widetilde \Psi}(x) dx =\bm 0_{2\times2}$.

Finally, for $n\ge 0,$ and using \eqref{qtilde-qhat}, \eqref{psitilde}, \eqref{relacionq-qhat}, \eqref{relacionqhat-q} and the orthogonality properties, we have 
\begin{align*}
\int_{-1}^{1}\widetilde{\bm Q}_n(x) \bm{\widetilde \Psi}(x) \bm{\widetilde Q}_n^T(x)dx&=\int_{-1}^{1} \bm {U}_n (x) S_0^{-1} S_0 \bm{\Psi_U}(x) S_0^T S_0^{-T}  \bm{U}_n^T(x)dx\\
&=\int_{-1}^{1}Y_n^{-1}\left[x\bm Q_n(x)-X_n\bm{U}_{n+1}(x) \right] \bm{\Psi_U}(x) \bm{U}_n^T(x)dx\\
&=Y_n^{-1}\int_{-1}^{1}x\bm Q_n(x)\bm{\Psi_U}(x) \bm{U}_n^T(x)dx\\
&=Y_n^{-1}\int_{-1}^{1}\bm Q_n(x)\bm \Psi(x)\left[R_n \bm Q_{n-1}(x)+S_n\bm Q_n(x) \right]^Tdx\\
&=Y_n^{-1}\left[\int_{-1}^{1}\bm Q_n(x)\bm \Psi(x) \bm Q_n^T(x)dx \right]S_n^T\\
&=Y_n^{-1} \Pi_n^{-1} S_n^T=(S_n^{-T}\Pi_n Y_n)^{-1}=\widetilde \Pi_n^{-T}=\widetilde \Pi_n^{-1},
\end{align*}
where in the last steps we have used the formula \eqref{coefpot1} and the fact that $(\widetilde \Pi_n)_{n\geq0}$ are diagonal matrices.
\end{proof}

\begin{remark}
Observe that the derivation of the spectral matrix $\bm{\widetilde \Psi}(x)$ for the matrix-valued polynomials $(\widetilde{\bm Q}_n)_{n\geq0}$ is not restricted to the case where we start with a birth-death chain on $\ZZ$ with transition probability matrix $P$. For \emph{any} tridiagonal block matrix $\bm P$ as in \eqref{P2} with $N\times N$ blocks, if we are able to find a factorization of the form $\bm P=\bm P_R\bm P_A$ where $\bm P_R$ and $\bm P_A$ are given by \eqref{BPRA}, then we can follow the same steps to compute the spectral matrix associated with the Darboux transformation  $\widetilde{\bm P}=\bm P_A\bm P_R$. Assuming that we have computed the spectral matrix $\bm\Psi(x)$ associated with $\bm P$ (which is not always possible, see Theorem 2.1 of \cite{DRSZ}) and that $\bm M_{-1}=\int_{-1}^{1}x^{-1}\bm\Psi(x)dx$ is well-defined (entry by entry), the spectral matrix $\bm{\widetilde \Psi}(x)$ associated with $\bm{\widetilde P}$ will be given by
$$
\bm{\widetilde \Psi}(x)=S_0\left(\frac{\bm\Psi(x)}{x}+\left[(\Pi_0Y_0S_0)^{-1}-\bm M_{-1}\right]\delta_0(x)\right)S_0^T.
$$ 
This is what is called a \emph{Geronimus transformation} of the original spectral matrix $\bm\Psi(x)$. Observe that $\bm{\widetilde \Psi}(x)$ is not necessarily a proper weight matrix even if $\bm{ \Psi}(x)$ is. For that we need the matrix $(\Pi_0Y_0S_0)^{-1}-\bm M_{-1}$ to be a positive semi-definite matrix. In the case we are treating in this paper we have that $\Pi_0Y_0S_0$ is a positive definite matrix (see \eqref{psihat}).
\end{remark}

\subsection{Example: random walk on $\ZZ$}

Let us consider an irreducible birth-death chain on $\ZZ$ with transition probability matrix $P$ as in \eqref{P1} with constant transition probabilities
$$
a_n=a, \quad b_n=b, \quad c_n=c, \quad n\in\ZZ,\quad a+b+c=1,\quad a,c>0,\quad b\geq0.
$$
This discrete-time birth-death chain is usually called a \emph{random walk}. If we consider an RA factorization we have that the continued fractions \eqref{ccff} can be explicitly computed and in this case they are given by
$$
H=\frac{1}{2}\left( 1+a-c-\sqrt{(1+c-a)^2-4c}\right), \quad H'=\frac{1}{2}\left( 1+c-a-\sqrt{(1+c-a)^2-4c}\right),
$$
with $a\le (1-\sqrt{c})^2$ (to ensure convergence). Following Proposition \ref{propalfx0} we have that the RA stochastic factori-zation is possible if and only if we take $\alpha \ge H'$ and $x_0\ge H$ bearing in mind that $\alpha+x_0\leq1$. Observe that if we choose $\alpha=H'$ and $x_0=H$, then we get 
$$
s_{-n}=H, \quad r_{-n}=H, \quad y_{-n}=1-H', \quad x_{-n}=H', \quad n\ge 1, 
$$
$$
s_n=1-H', \quad r_n=H', \quad y_n=1-H, \quad x_n=H, \quad n\ge 1, 
$$
and $y_0=1-H-H'$, $s_0=1$ and $r_0=0$. In this case the coefficients of the Darboux transformation remain ``almost" invariant, i.e., $\tilde a_n=a$, $\tilde b_n=b$ and $\tilde c_n=c$ for all $n$ except for $\tilde a_{-1}$, $\tilde a_0$, $\tilde b_0$, $\tilde c_{0}$, $\tilde c_1$.\\

The spectral matrix associated with this example is given by only an absolutely continuous part, i.e.,
\begin{equation}\label{WW1}
\bm\Psi(x)=\frac{1}{\pi\sqrt{(x-\sigma_-)(\sigma_+-x)}}\begin{pmatrix} 1 & \D\frac{x-b}{2c}\\
\D\frac{x-b}{2c}& a/c\end{pmatrix},\quad x\in[\sigma_-,\sigma_+],\quad \sigma_{\pm}=1-\left(\sqrt{a}\mp\sqrt{c}\right)^2.
\end{equation}
For details on how to compute this spectral matrix see Section 4.1 of \cite{dIJ}. Also a straightforward computation shows that the moment $\bm M_{-1}$ of $\bm\Psi$ is given by
\begin{equation*}\label{Mm11}
\bm M_{-1}=\begin{pmatrix} \D\frac{1}{\sqrt{\sigma_-\sigma_+}} &\D\frac{1}{2c}\left(1-\D\frac{b}{\sqrt{\sigma_-\sigma_+}}\right)\\[0.3cm]
\D\frac{1}{2c}\left(1-\D\frac{b}{\sqrt{\sigma_-\sigma_+}}\right)& \D\frac{a}{c\sqrt{\sigma_-\sigma_+}}\end{pmatrix}.
\end{equation*}
In order for $\bm M_{-1}$ to be well-defined we need to assume that $\sigma_->0$, i.e., $\sqrt{a}+\sqrt{c}<1$, or, in other words $a<(1-\sqrt{c})^2$, which is the condition for convergence of the continued fractions $H$ and $H'$. Once we have this information it is possible to compute the spectral matrix associated with the Darboux transformation $\widetilde P=P_AP_R$ at the beginning of Section \ref{sec3}, which we recall it is an ``almost'' birth-death chain except for the states $1$ and $-1$ (see \eqref{DTabc} and \eqref{DTdd}) and two free parameters, $\alpha$ and $x_0$. In this case we have
$$
S_0=\begin{pmatrix} 1&0\\1-c/\alpha& c/\alpha \end{pmatrix}.
$$
Following Theorem \ref{thmp} we have that the spectral matrix associated with $\widetilde P$ is given by (after some computations)
$$
\widetilde{\bm\Psi}(x)=\frac{1}{\pi x \sqrt{(x-\sigma_-)(\sigma_+-x)} } [\bm{\widetilde A}+ \bm{\widetilde B}x]+ \widetilde{\bm M}_{-1} \delta_0,
$$
where
$$
\bm{\widetilde A}=\begin{pmatrix} 1&\D\frac{2\alpha+H-H'-1}{2\alpha}\\[0.3cm]
\D\frac{2\alpha+H-H'-1}{2\alpha}& \D\frac{(\alpha-H')(H+\alpha-1)}{\alpha^2}\end{pmatrix}, \quad  \bm{\widetilde B}=\frac{1}{2\alpha}\begin{pmatrix} 0& 1\\ 1 & \D\frac{2(\alpha-c)}{\alpha} \end{pmatrix},
$$
$$
\widetilde{\bm M}_{-1}=\begin{pmatrix} \D\frac{x_0-H+\alpha-H'}{y_0(1-H-H')} &\D\frac{H'-\alpha}{\alpha(1-H-H')}\\[0.3cm]
\D\frac{H'-\alpha}{\alpha(1-H-H')} & -\D\frac{(\alpha-H')(1-H-\alpha)}{\alpha^2(1- H - H')}\end{pmatrix}.
$$
Observe that in the case that $\alpha=H'$ and $x_0=H$ we have that $\widetilde{\bm{M}}_{-1}=\bm 0_{2\times2}.$

\section{Absorbing-reflecting factorization}\label{secAR}

In this section we will use the same notation as in Section \ref{secRA}, but replacing all parameters, matrices, etc. in the RA factorization by a tilde superscript. For instance $y_n,x_n,s_n,r_n,\alpha$ in \eqref{PRR} and \eqref{PAA} will be replaced by $\tilde y_n,\tilde x_n,\tilde s_n,\tilde r_n,\tilde\alpha$, respectively, $\bm P_R, \bm P_A, Y_n, X_n, S_n, R_n$ in \eqref{BPRA} will be replaced by $\widetilde{\bm P}_R, \widetilde{\bm P}_A, \widetilde Y_n, \widetilde X_n, \widetilde S_n, \widetilde R_n$, respectively, and so on. 

As we saw at the beginning of Section \ref{sec3} the multiplication of matrices of the form $\widetilde P_A\widetilde P_R$, where $\widetilde P_A$ and $\widetilde P_R$ are given by \eqref{PAA} and \eqref{PRR}, gives rise to Markov chain which is ``almost'' a birth-death chain, except for the states $1$ and $-1$ where it is possible to go from $1$ to $-1$ and viceversa (see \eqref{DTdd}). Hence it will not be interesting to start from a birth-death chain $P$ and consider an absorbing-reflecting (or AR) factorization of the form $P=\widetilde P_A\widetilde P_R$, since this will imply that both $\widetilde P_A$ and $\widetilde P_R$ will be splited into two separated birth-death chains at the state $0$. Therefore it will make more sense to start with an irreducible Markov chain $\{X_t : t=0,1,\ldots\}$ on $\ZZ$ with transition probability matrix
\begin{equation}\label{P12}
P=\left(
\begin{array}{cccc|ccccc}
\ddots&\ddots&\ddots&&&&&\\
&c_{-2}&b_{-2}&a_{-2}&&&&&\\
&&c_{-1}&b_{-1}&a_{-1}&d_{-1}&&&\\
\hline
&&&c_0&b_0&a_0&&&\\
&&&d_1&c_1&b_1&a_1&&\\
&&&&&c_2&b_2&a_2&\\
&&&&&&\ddots&\ddots&\ddots
\end{array}
\right).
\end{equation}
A diagram of the transitions of this Markov chain is similar to the one given in Section \ref{sec3}. If we perform the same labeling as in \eqref{label} then $P$ is equivalent to a semi-infinite $2\times2$ block tridiagonal matrix $\bm P$ of the form
\begin{equation}\label{P22}
\bm P=
\left(
\begin{array}{ccccccccccccc}
b_0&c_0&\temp &\hspace{-0.3cm} a_0&0&\temp & & & \\
a_{-1}&b_{-1}&\temp &\hspace{-0.3cm}d_{-1}&c_{-1}&\temp & & & \\
\cline{1-8}
c_1&d_1&\temp &\hspace{-0.3cm}b_1&0&\temp &\hspace{-0.3cm} a_1&0&\temp& & \\
0&a_{-2}&\temp &\hspace{-0.3cm}0&b_{-2}&\temp &\hspace{-0.3cm}0&c_{-2}&\temp& & \\
\cline{1-11}
&&\temp &\hspace{-0.3cm}c_2&0&\temp &\hspace{-0.3cm}b_2&0&\temp&\hspace{-0.3cm} a_2 &0&\temp \\
&&\temp &\hspace{-0.3cm}0&a_{-3}&\temp &\hspace{-0.3cm}0&b_{-3}&\temp&\hspace{-0.3cm} 0 & c_{-3}&\temp \\
\cline{3-13}
&&&&&\temp &\hspace{-0.3cm}\ddots&&\temp &\hspace{-0.3cm} \ddots&&\temp&\hspace{-0.3cm} \ddots \\
\end{array}
\right),
\end{equation}
where the only difference with the coefficients in \eqref{ABC} is the triangular shape of the matrices $A_0$ and $C_1$, given by
$$
A_0=\begin{pmatrix} a_0& 0 \\ d_{-1}& c_{-1}\end{pmatrix}, \quad C_1=\begin{pmatrix} c_1& d_1 \\ 0 &  a_{-2} \end{pmatrix}.
$$
Now, using the same notation as in \eqref{BPRA} (see also \eqref{XYSR}), let us consider the block matrix factorization $\bm P=\widetilde{\bm P}_A \widetilde{\bm P}_R$, which in this case is a LU block matrix factorization. A direct computation shows
\begin{equation}\label{ABC-tildeXYSR}
\begin{split}
A_n&=\widetilde S_{n}\widetilde X_n, \quad n\ge 0, \\
B_n&=\widetilde S_n\widetilde Y_n+\widetilde R_{n}\widetilde X_{n-1} , \quad n\ge1,\quad B_0=\widetilde S_0 \widetilde Y_0,  \\
C_n&=\widetilde R_n\widetilde Y_{n-1} , \quad n\ge 1, 
\end{split}
\end{equation}
or equivalently, using $P=\widetilde{P}_A \widetilde{P}_R$, we obtain
\begin{align}
&a_n=\tilde x_n\tilde s_{n}, \quad n\ge 1,\quad a_0=\tilde x_0,\quad a_{-n}=\tilde y_{-n+1}\tilde r_{-n},\quad n\ge 1,\label{aatilde}\\
&b_n=\tilde y_n\tilde s_n+\tilde x_{n-1} \tilde r_{n}, \quad n\ge 1, \quad b_0=\tilde  y_0,\quad b_{-n}=\tilde y_{-n} \tilde s_{-n}+ \tilde x_{-n+1} \tilde r_{-n},\quad n\ge 1,\label{bbtilde}\\
&c_n=\tilde y_{n-1} \tilde r_n, \quad n\ge 1, \quad c_0=\tilde \alpha, \quad c_{-n}=\tilde x_{-n} \tilde s_{-n},\quad n\ge 1,\label{cctilde}\\
&d_{-1}=\tilde r_{-1} \tilde x_0,\quad d_{1}=\tilde r_{1} \tilde \alpha, \label{ddtilde}
\end{align}
with  $\widetilde P_A$ and $\widetilde P_R$ stochastic matrices, i.e., all entries are nonnegative and
\begin{equation}\label{x0stilde}
\tilde \alpha+\tilde x_0+\tilde y_0=1,\quad \tilde x_n+\tilde y_n=1, \quad \tilde s_n+\tilde r_n=1, \quad n\in\ZZ \setminus \{0\}.
\end{equation}
Observe that, from \eqref{ddtilde} and \eqref{aatilde}, we have $\tilde r_{-1}=d_{-1}/\tilde x_0=a_{-1}/\tilde y_0$ and $\tilde r_1=d_1/\tilde\alpha=c_{1}/\tilde y_0$. Therefore, from \eqref{aatilde} and \eqref{bbtilde}, we get that the factorization is posible if and only if
\begin{equation}\label{conddd}
d_{-1}=\frac{a_{-1}a_0}{b_0},\quad d_1=\frac{c_0c_1}{b_0}.
\end{equation}
This means that, in order to have a stochastic AR factorization, the Markov chain \eqref{P12} is not a general one but restricted to the conditions \eqref{conddd}. Additionally we need to have that $0<d_{-1},d_1<1$, so we have to assume that $b_0>\max\{a_{-1}a_0,c_0c_1\}$. Once we have all the previous considerations it is possible to get the coefficients $\tilde y_0$, $\tilde r_1$, $\tilde s_1$, $\tilde x_1$, $\tilde y_1,\ldots$ recursively using \eqref{bbtilde}, \eqref{cctilde} and \eqref{x0stilde} in that order. On the other hand it is possible to get $\tilde y_0$, $\tilde r_{-1}$, $\tilde s_{-1}$, $\tilde x_{-1}$, $\tilde y_{-1}$, $\tilde r_{-2},\ldots$ recursively using \eqref{bbtilde}, \eqref{aatilde} and \eqref{x0stilde} in that order. In this case we also have that $\tilde x_0=a_0$ and $\tilde \alpha=c_0$. Therefore \emph{there is no free parameter} and the factorization is unique. Let $\widetilde H$ and $\widetilde H'$ be the following continued fractions 
\begin{equation}\label{ccff2}
\widetilde H=\cFrac{c_1}{1}-\cFrac{a_1}{1}-\cFrac{c_2}{1}-\cFrac{a_2}{1}-\cdots,\quad \widetilde H'=\cFrac{a_{-1}}{1}-\cFrac{c_{-1}}{1}-\cFrac{a_{-2}}{1}-\cFrac{c_{-2}}{1}-\cdots
\end{equation}
For each continued fraction, consider the corresponding sequence of convergents
$(\tilde h_{n})_{n\geq0}$ and $(\tilde h_{-n}')_{n\geq0}$, given by
\begin{equation}\label{hhh2}
\tilde h_{n}=\frac{\tilde A_n}{\tilde B_n},\quad \tilde h_{-n}'=\frac{\tilde A_{-n}'}{\tilde B_{-n}'}.
\end{equation}
Again, we refer to \cite{GdI2, dIJ} to find more information about the notation and definitions on continued fractions.
\begin{proposition}\label{Propo2}
Let $\widetilde H$ and $\widetilde H'$ be the continued fractions given by \eqref{ccff2} and the corresponding convergents $\tilde h_n$ and $\tilde h_{-n}$ defined by \eqref{hhh2}. Assume that 
\begin{equation*}\label{AnBn2}
0<A_n<B_n,\quad\mbox{and}\quad0<A_{-n}'<B_{-n}',\quad n\geq1.
\end{equation*}
Then both $\widetilde H$ and $\widetilde H'$ are convergent. Moreover, let $P=\widetilde P_A\widetilde P_R$. Then, both $\widetilde P_A$ and $\widetilde P_R$ are stochastic matrices if and only if
\begin{equation*}\label{yy0r2}
b_0>\max\{\widetilde H,\widetilde H'\}.
\end{equation*}
\end{proposition}
\begin{proof}
The proof follows the same lines as the proof of Theorem 2.1 of \cite{dIJ} and will be omitted (see also Theorem 2.1 of \cite{GdI3}).
\end{proof}
\begin{remark}
Observe that the condition of the previous proposition implies that $b_0\ge \widetilde H$ and $b_0\ge \widetilde H'$. Also, if $b_0\ge \widetilde H$, then in particular $b_0> c_1$ and since we know that $0<c_0<1$, then we have $b_0>c_1>c_0c_1$.
Similarly, if $b_0\ge \widetilde H'$, in particular we have $b_0>a_{-1}$ and since $0<a_0<1$, then we get 
$b_0>a_{-1}>a_0a_{-1}$. This allows us  to conclude that if $b_0 \ge \max\{ \widetilde H, \widetilde H'\}$, then $b_0\ge \max\{a_{-1}a_0, c_0c_1\}$ and therefore $0<d_1, d_{-1}<1$. Observe that this does not mean that the factorization is always possible since we need to have \eqref{conddd}. 
\end{remark}


\subsection{Stochastic Darboux transformation and the associated spectral matrix}

As we did in Section \ref{sec3}, if $P=\widetilde{P}_A \widetilde{P}_R$ (or equivalently $\bm P=\widetilde{\bm P}_A \widetilde{\bm P}_R$) then we can perform a Darboux transformation given by $\widehat{P}=\widetilde{P}_R \widetilde{P}_A$ (or $\widehat{\bm P}=\widetilde{\bm P}_R \widetilde{\bm P}_A$). In block matrix form we have
\begin{equation}\label{hatp}
\widehat{\bm P}=\begin{pmatrix}
\widehat B_0 &\widehat  A_0 & & \\
\widehat  C_1& \widehat B_1& \widehat A_1& \\
& \widehat  C_2& \widehat B_2 & \widehat  A_2\\
& &\ddots &\ddots&\ddots \\
\end{pmatrix}=\begin{pmatrix}
\widetilde Y_0 &\widetilde X_0 & & \\
& \widetilde Y_1&\widetilde X_1& \\
& & \widetilde Y_2 &\widetilde  X_2\\
& & &\ddots&\ddots \\
\end{pmatrix}
\begin{pmatrix}
\widetilde S_0 & & & \\
\widetilde R_1& \widetilde S_1&& \\
& \widetilde R_2 & \widetilde S_2 &\\
& & \ddots &\ddots \\
\end{pmatrix}.
\end{equation}
A direct computation shows
\begin{equation}\label{DTABC2}
\begin{split}
\widehat  A_n&= \widetilde X_n\widetilde S_{n+1}, \quad n\ge 0, \\
\widehat  B_n&=\widetilde X_{n}\widetilde R_{n+1}+\widetilde Y_n\widetilde S_n, \quad n\ge 0, \\
\widehat  C_n&=\widetilde Y_{n}\widetilde R_n, \quad n\ge 1. 
\end{split}
\end{equation}
An important difference now is that $\widehat{P}$ is in this case a discrete-time birth-death chain on $\ZZ$ (without transitions between the states $1$ and $-1$). As in Section \ref{sec3} we are interested in the spectral matrix associated with $\widehat{\bm P}$ given that we have information about the spectral matrix associated with $\bm P$. In principle we can not guarantee that there exists a weight matrix associated with $\bm P$. But it is possible to see, using Theorem 2.1 of \cite{DRSZ}, that there exists a spectral matrix $\bm\Psi$ such that the polynomials $(\bm Q_n)_{n\geq0}$ defined by the three-term recurrence relation \eqref{ttrrq} are orthogonal with respect to the spectral matrix $d\bm\Psi(x)$. The sequence of nonsingular matrices $(R_n)_{n\geq0}$ in that theorem is given by
\begin{equation*}
R_n=\begin{pmatrix} \sqrt{\pi_n} & 0 \\ 0 & \sqrt{\pi_{-n-1}}\end{pmatrix},\quad n\geq0,
\end{equation*}
where $\pi=(\pi_n)_{n\in\ZZ}$ are the potential coefficients given by \eqref{potcoeff}. A direct computation using \eqref{conddd} shows that $R_nB_nR_n^{-1}, n\geq0,$ are always symmetric matrices and that $R_n^TR_n=(C_1^T\cdots C_n^T)^{-1}R_0^TR_0 A_0 \cdots A_{n-1}, n\ge 1.$ Therefore we have $R_n^TR_n=\Pi_n$ where $\Pi_n$ is defined by \eqref{potcoeffm}.

\medskip

Consider now the matrix-valued polynomials
\begin{equation*}
\bm{\bar Q}_n(x)=\widetilde Y_n \bm Q_{n}(x)+\widetilde X_n \bm Q_{n+1}(x), \quad n \ge 0,
\end{equation*}
where $(\widetilde Y_n)_{n\geq0}$ and $(\widetilde X_n)_{n\geq0}$ are defined by \eqref{XYSR}. If we denote $\bm{\bar Q}=(\bm{\bar Q}_0^T,\bm{\bar Q}_1^T,\cdots)^T$, then we have that $\bm{\bar Q}=\bm{\widetilde P}_R\bm Q$. Therefore, from the AR factorization of $\bm P$, we get $\bm{\widetilde P_A}\bm{\bar Q}= \bm{\widetilde P_A} \bm{\widetilde P_R} \bm Q= \bm P\bm Q=x\bm Q$, or in other words
\begin{equation*}
\begin{split}
x\bm Q_0(x)&=\widetilde S_0 \bm{\bar Q}_0(x),\\
x\bm Q_n(x)&=\widetilde R_n \bm{\bar Q}_{n-1}(x)+\widetilde S_n\bm{\bar Q}_{n}(x), \quad n\ge 1.
\end{split}
\end{equation*}
From the previous equation we have that $\bm{\bar Q}_0(x)=x \widetilde S_0^{-1}\bm Q_0(x)=x \widetilde S_0^{-1}$ and by induction we can prove that $\bm{\bar Q}_n(x)=x\bm T_n(x)$, where $(\bm T_n)_{n\geq0}$ is a family of matrix-valued polynomials with $\mbox{deg}(\bm T_n(x))=n, n\geq0,$ and nonsingular leading coefficient. We can rewrite the previous two formulas in terms of the polynomials $(\bm T_n)_{n\geq0}$. Indeed,
\begin{equation}\label{relaciont-q}
x \bm{T}_n(x)=\widetilde Y_n \bm Q_{n}(x)+\widetilde X_n \bm Q_{n+1}(x), \quad n \ge 0,
\end{equation}
and 
\begin{equation}\label{relacionq-t}
\begin{split}
\bm Q_0(x)&=\widetilde S_0 \bm{T}_0(x),\\
\bm Q_n(x)&=\widetilde R_n \bm{T}_{n-1}(x)+\widetilde S_n\bm{T}_{n}(x), \quad n\ge 1.
\end{split}
\end{equation}
Since $\bm{T}_0(x)=\widetilde S_0^{-1}\neq I_{2\times2}$, let us define a new family 
$(\bm{\widehat Q}_n)_{n\geq0}$ such that $\bm{\widehat Q}_0=I_{2\times2}$, i.e.,
\begin{equation}\label{widehatq}
\bm{\widehat Q}_n(x)=\bm{T}_n(x) \widetilde S_0.
\end{equation}
Finally, if we define $(\widehat \Pi_n)_{n\geq0}$ as the solution of the symmetry equations for $\bm{\widehat P}$ in \eqref{hatp}, given by 
\begin{equation*}
\widehat \Pi_n=(\widehat C_1^T\cdots \widehat C_n^T)^{-1}\widehat \Pi_0 \widehat A_0 \cdots \widehat  A_{n-1}, \quad n\ge 1.
\end{equation*}
Then, using \eqref{ABC-tildeXYSR}, \eqref{DTABC2} and the previous equation, we get
\begin{equation}\label{coefpot2}
\widehat \Pi_n=\widetilde Y_n^{-T} \Pi_n \widetilde S_n, \quad n\ge 0.
\end{equation}
Computing $\widehat\Pi_0$ using \eqref{aatilde} and \eqref{cctilde}, we have
\begin{equation*}
\widehat \Pi_0=\frac{1}{\tilde y_0} \begin{pmatrix}  1 &0\\ 0& \frac{\widetilde \alpha \widetilde s_{-1}}{\widetilde y_{-1} \widetilde r_{-1}}\end{pmatrix}.
\end{equation*}
Therefore, as in the RA factorization, $(\widehat \Pi_n)_{n\geq0}$ are always \emph{diagonal matrices}. We are ready to prove the main result of this section.
\begin{theorem}\label{thmp2}
Let $\{X_t: t=0, 1, \dots\}$ be the Markov chain on $\ZZ$ with transition probability matrix $P$ given by \eqref{P12} and $\{\widetilde X_t : t=0, 1, \dots\}$ the birth-death chain generated by the Darboux transformation of $P=\widetilde P_A\widetilde P_R$. Then the matrix-valued polynomials $(\widehat{\bm Q}_n)_{n\geq0}$ defined by \eqref{widehatq} are orthogonal with respect to the following spectral matrix 
\begin{equation}\label{hatpsi1}
\bm{\widehat \Psi}(x)=x\widetilde S_0^{-1} \bm \Psi(x)\widetilde S_0^{-T},
\end{equation}
where the constant matrix $\widetilde S_0$ is defined by \eqref{XYSR} and $\bm \Psi(x)$ is the original spectral matrix associated with $P$. Moreover, we have
\begin{equation*}\label{qusnorms2}
\int_{-1}^{1} \widehat{\bm Q}_n (x) \widehat{\bm\Psi} (x) \widehat{\bm Q}_m^T(x) dx=\widehat{\Pi}_n^{-1}\delta_{n,m},
\end{equation*}
where $(\widehat{\Pi}_n)_{n\geq0}$ are defined by \eqref{coefpot2}.
\end{theorem}
\begin{proof}
For $n\ge 1$ and $j=0, \dots, n-1,$ we have 
\begin{equation*} 
\begin{split}
\int_{-1}^{1} \bm{\widehat Q}_n(x) \bm{\widehat \Psi}(x) x^j dx &= \int_{-1}^{1}x\bm T_n(x)\bm \Psi(x)x^{j}\widetilde S_0^{-T}dx =\int_{-1}^{1}[\widetilde Y_n \bm Q_n(x)+\widetilde X_n \bm Q_{n+1}(x)] \bm \Psi(x)x^{j}\widetilde S_0^{-T}dx\\
&=\widetilde Y_n \int_{-1}^{1}\bm Q_n(x) \bm \Psi(x)x^{j}\widetilde S_0^{-T}dx +\widetilde X_n \int_{-1}^{1} \bm Q_{n+1}(x)\bm \Psi(x)x^{j}\widetilde S_0^{-T}dx =\bm 0_{2\times2},
\end{split}
\end{equation*}
where for the first equality we have used \eqref{widehatq} and \eqref{hatpsi1}, for the second equality we have used \eqref{relaciont-q}, and finally we have used the orthogonality of the family $(\bm Q_n)_{n\geq0}$. Finally, for $n\ge 0,$ and using \eqref{relaciont-q}, \eqref{relacionq-t}, \eqref{coefpot2} and the orthogonality properties, we have
\begin{equation*}
\begin{split}
\int_{-1}^{1} \bm{\widehat Q}_n(x) \bm{\widehat \Psi}(x) \bm{\widehat Q}_n^T(x)dx&= \int_{-1}^{1}\bm T_n(x) x \bm \Psi(x)\bm T_n^T(x)dx=\int_{-1}^{1}[\widetilde Y_n \bm Q_n(x)+\widetilde X_n \bm Q_{n+1}(x)] \bm \Psi(x)\bm T_n^T(x)dx\\
&=\widetilde Y_n\int_{-1}^{1} \bm Q_n(x)\bm \Psi(x)\bm T_n^T(x)dx= \widetilde Y_n\int_{-1}^{1} \bm Q_n(x)\bm \Psi(x)[\widetilde S_n^{-1}\bm Q_n(x)-\widetilde S_n^{-1}\widetilde R_n \bm T_{n-1}(x)]^Tdx\\
&=\widetilde Y_n\int_{-1}^{1} \bm Q_n(x)\bm \Psi(x)\bm Q_{n}^T(x)\widetilde S_n^{-T}dx=\widetilde Y_n \Pi_n^{-1}\widetilde S_n^{-T}= (\widetilde S_n^{T}\Pi_n\widetilde Y_n^{-1})^{-1}=\widehat \Pi_n^{-T}= \widehat \Pi_n^{-1},
\end{split}
\end{equation*}
where in the final step we have used the fact that $(\widehat \Pi_n)_{n\geq0}$ are diagonal matrices.
\end{proof}
\begin{remark}
The weight matrix $\widehat{\bm\Psi}(x)$ in \eqref{hatpsi1} is called a \emph{Christoffel transformation} of the original spectral matrix $\bm\Psi(x)$.
\end{remark}

\subsection{Example: random walk with transitions between the states $1$ and $-1$}

Let us consider an irreducible Markov chain on $\ZZ$ with transition probability matrix $P$ as in \eqref{P12} where
$$
a_n=a, \quad n\in\ZZ\setminus\{-1\},\quad b_n=b, \quad n\in\ZZ,\quad c_n=c, \quad n\in\ZZ\setminus\{1\},
$$
and as usual $a+b+c=1,a,c>0,b\geq0$. From \eqref{conddd} the values of $a_{-1},d_{-1},c_1,d_1$ must be given in terms of $a,b,c$. Indeed,
$$
a_{-1}=\frac{ab}{1-c},\quad d_{-1}=\frac{a^2}{1-c},\quad c_1=\frac{bc}{1-a},\quad d_1=\frac{c^2}{1-a}.
$$
Since $1-c>0$ and $1-a>0,$ this implies that $a_{-1}, d_{-1}, c_1, d_1>0$. Now, observe that $a+c\le1$ and $ab<a$ since $0\leq b<1$. Then we have $ab+c<a+c\le1$, which implies that $ab<1-c$ and therefore $a_{-1}<1$. In the same way but using $a^2<a$, $bc<c$ and $c^2<c$ we get $d_{-1}<1$, $c_1<1$ and $d_1<1,$ respectively. Therefore, independently of the choice of $a, b$ and $c$, $P$ is always a stochastic matrix.

If we consider the AR factorization we have that the continued fractions in \eqref{ccff2} can be explicitly computed and in this case they are given by
$$
\widetilde H=\frac{bc}{J(1-a)}, \quad \widetilde H'=\frac{ab}{J'(1-c)},
$$ 
where 
\begin{equation}\label{JJp}
J=\frac{1}{2} \left( 1+c-a+\sqrt{(1+c-a)^2-4c}\right),\quad J'=\frac{1}{2}\left( 1+a-c+\sqrt{(1+c-a)^2-4c} \right).
\end{equation}
Therefore we get
$$
\widetilde H=\frac{b}{2(1-a)}\left( 1+c-a+\sqrt{(1+c-a)^2-4c} \right),\quad 
\widetilde H'=\frac{b}{2(1-c)}\left( 1+a-c+\sqrt{(1+c-a)^2-4c} \right),
$$
with $a\le (1-\sqrt{c})^2$ (to ensure convergence). With this condition we immediately have that $b>\widetilde H, b>\widetilde H'$ and then $b>\max \{\widetilde H, \widetilde H' \}$. Therefore, according to Proposition \ref{Propo2}, we have that the AR stochastic  factorization of $P$ is always possible. 


Let us now compute the spectral matrix $\bm\Psi(x)$ associated with this example. Since $P$ is not tridiagonal we do not have a birth-death chain on the integers $\ZZ$, so we can not apply the same methodology as we did in the previous section. However we can consider the block tridiagonal structure of $P$ given by $\bm P$ in \eqref{P22} and use the theory of matrix-valued orthogonal polynomials to compute the spectral matrix $\bm\Psi(x)$. $\bm P$ is given in this case by
\begin{equation*}
\bm P=
\left(
\begin{array}{ccccccccccccc}
b&c&\temp &\hspace{-0.3cm} a&0&\temp & & & \\
\frac{ab}{1-c}&b&\temp &\hspace{-0.3cm}\frac{a^2}{1-c}&c&\temp & & & \\
\cline{1-8}
\frac{bc}{1-a}&\frac{c^2}{1-a}&\temp &\hspace{-0.3cm}b&0&\temp &\hspace{-0.3cm} a&0&\temp& & \\
0&a&\temp &\hspace{-0.3cm}0&b&\temp &\hspace{-0.3cm}0&c&\temp& & \\
\cline{1-11}
&&\temp &\hspace{-0.3cm}c&0&\temp &\hspace{-0.3cm}b&0&\temp&\hspace{-0.3cm} a &0&\temp \\
&&\temp &\hspace{-0.3cm}0&a&\temp &\hspace{-0.3cm}0&b&\temp&\hspace{-0.3cm} 0 & c&\temp \\
\cline{3-13}
&&&&&\temp &\hspace{-0.3cm}\ddots&&\temp &\hspace{-0.3cm} \ddots&&\temp&\hspace{-0.3cm} \ddots \\
\end{array}
\right).
\end{equation*}
Let us now use Theorem 2.1 of \cite{Clay} to obtain a relation between the Stieltjes transform of the spectral matrix $\bm\Psi$, given by definition by $B(\bm\Psi;z)=\int_{-1}^1(z-x)^{-1}d\bm\Psi(x)$, and the Stieltjes transform of the spectral matrix $\bm\Psi_0$ of the \emph{0-th associated process} $\bm P_0$, which is constructed by removing the first block row and column of $\bm P$. Observe that $\bm P_0$ has constant diagonal block entries. Therefore, using Theorem 2.1 of \cite{Clay}, it is easy to compute the corresponding Stieltjes transform $B(\bm\Psi_0;z)$, given by
$$
B(\bm\Psi_0;z)=\begin{pmatrix} \D\frac{z-b\pm\sqrt{(z-\sigma_+)(z-\sigma_-)}}{2ac} & 0 \\ 0 &\D\frac{z-b\pm\sqrt{(z-\sigma_+)(z-\sigma_-)}}{2c^2} \end{pmatrix},\quad z\in\mathbb{C}\setminus[\sigma_-,\sigma_+],
$$
where $\sigma_{\pm}$ are given by \eqref{WW1}. Now using again Theorem 2.1 of \cite{Clay} we have that the Stieltjes transform $B(\bm\Psi;z)$ of $\bm P$ satisfies the algebraic equation
\begin{equation}\label{sttra}
B(\bm\Psi;z)\Pi_{\bm\Psi}\left[zI_{2\times2}-B_0-A_0B(\bm\Psi_0;z)\Pi_{\bm\Psi_0}C_1\right]=I_{2\times2},
\end{equation}
where
$$
A_0=\begin{pmatrix} a& 0 \\ \frac{a^2}{1-c}& c\end{pmatrix}, \quad B_0=\begin{pmatrix} b& c \\ \frac{ab}{1-c}& b\end{pmatrix}, \quad C_1=\begin{pmatrix} \frac{bc}{1-a}& \frac{c^2}{1-a} \\ 0 &  a \end{pmatrix},
$$
and $\Pi_{\bm\Psi}$ and $\Pi_{\bm\Psi_0}$ are the inverses of the 0-th norms of each spectral matrix, given in this case by
$$
\Pi_{\bm\Psi}=\begin{pmatrix} 1& 0 \\ 0& \frac{c(1-c)}{ab}\end{pmatrix},\quad \Pi_{\bm\Psi_0}=\begin{pmatrix} 1& 0 \\ 0& c/a\end{pmatrix}.
$$
Solving \eqref{sttra} and after some tedious but straightforward computations we have that
$$
B(\bm\Psi;z)=\begin{pmatrix} B(\psi_{11};z) & B(\psi_{12};z) \\ B(\psi_{12};z)  & B(\psi_{22};z) \end{pmatrix},
$$
where
$$
B(\psi_{ij};z)=\frac{p_{ij}(z)+q_{ij}(z)\sqrt{(z-\sigma_+)(z-\sigma_-)}}{r_{ij}(z)},\quad z\in\mathbb{C}\setminus[\sigma_-,\sigma_+],
$$
and $p_{ij}(z), q_{ij}(z), r_{ij}(z)$ are polynomials given by
\begin{align}
\nonumber p_{11}(z)&=2(1-a)(1-c)z^3-4b(1-a)(1-c)z^2+\gamma_{11}z-b^2((a-c)^2-a-c),\\
\nonumber q_{11}(z)&=b\left[-2(1-a)(1-c)z-a(1-a)-c(1-c)\right],\\
\nonumber r_{11}(z)&=2(1-a)(1-c)z^4-4b(1-a)(1-c)z^3+(\gamma_{11}-b^2(2ac-a-c))z^2+4ab^2cz-2ab^2c,\\
\nonumber p_{12}(z)&=b\left[-(1-a)(1-c)z^3+b(1-a)(2-3c)z^2+\gamma_{12}z-bc((1-c)^2-a(1+c))\right]\\
\label{quss}q_{12}(z)&=b\left[-(1-a)(1-c)z^2+(1-a)(1+2c^2-a-3c)z+bc(1-c)\right]\\
\nonumber r_{12}(z)&=cr_{11}(z),\\
\nonumber p_{22}(z)&=b^2\left[-(1-a)z^3+(1-a)(b+2(1-c))z^2+\gamma_{22}z-b(ac-c^2+a+2c-1)\right]\\
\nonumber q_{22}(z)&=b\left[-(1-a)(1-c+a)z^2+2b(1-a)(1-c)z-b^2(1-c)\right]\\
\nonumber r_{22}(z)&=cr_{11}(z),
\end{align}
where
\begin{align*}
\gamma_{11}&=2(1-a)^3+2(1-c)^3-2+2ac(2+a+c-4ac)+b^2(2ac-a-c),\\
\gamma_{12}&=a^3+a^2(2c^2+2c-3)+a(1-c)(2c^2-4c+3)-(1-c)^2(1-3c),\\
\gamma_{22}&=-2a^2(1+c)+a(2c^2-5c+5)-3(1-c)^2.
\end{align*}
Therefore the Stieltjes transform $B(\bm\Psi;z)$ can be written as
$$
B(\bm\Psi;z)=\frac{\sqrt{(z-\sigma_+)(z-\sigma_-)}}{cr_{11}(z)}\begin{pmatrix} cq_{11}(z) & q_{12}(z) \\ q_{12}(z)  & q_{22}(z) \end{pmatrix}+\frac{1}{cr_{11}(z)}\begin{pmatrix} cp_{11}(z)& p_{12}(z)\\p_{12}(z) & p_{22}(z)\end{pmatrix}.
$$
Observe that $r_{11}(z)$ is a polynomial of degree 4, so the Stieltjes transform may have at most 4 real poles. We have not been able to compute an explicit expression of these zeros, but if we assume that $c=a$ then it is possible to have an explicit expression of them. 

So let us assume from now on that $c=a$. Then $b=1-2a$ and the polynomial $r_{11}(z)$ has now a simpler expression:
$$
r_{11}(z)=2(1-z)(z(1-a)+a)[(a-1)z^2+b^2z-ab^2].
$$
The zeros of $r_{11}(z)$ are given by
\begin{equation}\label{zerr}
1,\quad-\frac{a}{1-a},\quad\frac{b(b\pm\sqrt{2b^2-1})}{1+b}.
\end{equation}
If $\sqrt{2}/2<b<1$, i.e., $0<a<(2-\sqrt{2})/4$, there can be at most 4 different real zeros. This means that the spectral matrix will consist of a continuous density plus possibly some Dirac delta masses located at these zeros with certain weights. Let us write the spectral matrix as $\bm\Psi(x)=\bm\Psi_c(x)+\bm\Psi_d(x)$. Using the Stieltjes-Perron inversion formula the continuous part of the spectral matrix is given by
\begin{equation}\label{weid}
\bm\Psi_c(x)=\frac{\sqrt{(\sigma_+-x)(x-\sigma_-)}}{c\pi r_{11}(x)}\begin{pmatrix} cq_{11}(x) & q_{12}(x) \\ q_{12}(x)  & q_{22}(x) \end{pmatrix},\quad x\in[\sigma_-,\sigma_+]=[1-4a,1],
\end{equation}
where
\begin{align*}
q_{11}(x)&=-2(1-2a)(1-a)(x(1-a)+a),\\
q_{12}(x)&=-(1-2a)(1-a)(x(1-a)+a)(x-1+2a),\\
q_{22}(x)&=-(1-2a)(1-a)(x^2-2(1-a)(1-2a)x+(1-2a)^2).
\end{align*}
The discrete masses come from the residues at the simple poles of $B(\bm\Psi;z)$, given by \eqref{zerr}. It is possible to see that these zeros (if real) are in $[-1,1]\setminus(\sigma_-,\sigma_+)$. A straightforward computation shows that all these discrete masses are identically $\bm 0_{2\times2}$. Therefore $\bm\Psi_d(x)=\bm 0_{2\times2}$ and $\bm\Psi(x)=\bm\Psi_c(x)$. As a consequence the spectral matrix only have a continuous part given by \eqref{weid}. In the general case (if $c\neq a$) we have extensive computational evidences that the spectral matrix has again only a continuous part, given by \eqref{weid}.

\medskip

Now let us consider the stochastic Darboux transformation of the AR factorization. As we said in Section \ref{secAR} (just before \eqref{ccff2}), there is no free parameter and the factorization is unique. The Darboux transformation is then given by $\widehat{\bm P}$ in \eqref{hatp} (in block matrix form), but if we consider it as a Markov chain on $\ZZ$, the matrix $\widehat{P}$ is now a discrete-time birth-death chain on $\ZZ$. The transition probabilities of this birth-death chain are highly nontrivial and they can be computed from \eqref{DTABC2}. Let us call $J$ the continued fraction
$$
J=\cFrac{c}{1}-\cFrac{a}{1}-\cFrac{c}{1}-\cFrac{a}{1}-\cdots
$$
and $j_n=\alpha_n/\beta_n$ the corresponding convergents. Observe that $J$ is convergent as long as $a\le (1-\sqrt{c})^2$ and the limit is given by \eqref{JJp}. The sequences $\alpha_n$ and $\beta_n$ of the convergents $j_n$ can be computed recursively using the following relations:
\begin{align*}
\alpha_{2n}&=\alpha_{2n-1}-a\alpha_{2n-2},\quad n\geq1,\quad \alpha_{2n+1}=\alpha_{2n}-c\alpha_{2n-1},\quad n\geq0,\quad \alpha_{-1}=1,\quad \alpha_0=1,\\
\beta_{2n}&=\beta_{2n-1}-a\beta_{2n-2},\quad n\geq1,\quad \beta_{2n+1}=\beta_{2n}-c\beta_{2n-1},\quad n\geq0,\quad \beta_{-1}=0,\quad \beta_0=1.
\end{align*}
The first few convergents are given by
$$
j_0=0,\quad j_1=c,\quad j_2=\frac{c}{1-a},\quad j_3=\frac{c(1-c)}{b},\quad j_4=\frac{cb}{b-a(1-a)},\quad j_5=\frac{c(b-c(1-c))}{b(1-c)-a(1-a)},\quad\ldots
$$
A straightforward computation using \eqref{aatilde} and \eqref{cctilde} shows that the coefficients $\tilde x_n,\tilde y_n, \tilde r_n, \tilde s_n, n\in\ZZ,$ can be written in terms of the convergents $j_n$. Indeed,
\begin{align*}
\tilde{x}_n&=\frac{a}{1-j_{2n}},\;n\geq0,\quad\tilde{y}_n=1-\tilde{x}_n,\;n\geq1,\quad\tilde{y}_0=b,\quad\tilde\alpha=c,\quad\tilde{x}_{-n}=j_{2n+1},\quad\tilde{y}_{-n}=1-j_{2n+1},\;n\geq1,\\
\tilde{r}_n&=j_{2n},\quad\tilde{s}_n=1-j_{2n},\;n\geq0,\quad\tilde{r}_{-n}=\frac{a}{1-j_{2n-1}},\quad\tilde{s}_{-n}=1-\tilde{r}_{-n},\;n\geq1.
\end{align*}
Therefore, the coefficients $\hat{a}_n,\hat{b}_n,\hat{c}_n, n\in\ZZ,$ of the birth-death chain $\widehat{P}$ are given by
\begin{align*}
\hat{a}_n=&\frac{a(1-j_{2n+2})}{1-j_{2n}},\quad n\geq0,\quad\hat{a}_{-n}=\frac{a(1-j_{2n+1})}{1-j_{2n-1}},\quad n\geq1,\\
\hat{c}_n=&\frac{j_{2n}(1-a-j_{2n})}{1-j_{2n}},\quad n\geq0,\quad\hat{c}_{-n}=\frac{j_{2n+1}(1-a-j_{2n+1})}{1-j_{2n+3}},\quad n\geq1,\\
\hat{b}_n=&1-\hat{a}_n-\hat{c}_n,\quad n\in\ZZ.\\
\end{align*}
Although these coefficients are highly nontrivial it is possible to compute the spectral matrix $\widehat{\bm\Psi}$ of the birth-death chain $\widehat{P}$ using Theorem \ref{thmp}. Indeed, the spectral matrix is given by
$$
\widehat{\bm\Psi}(x)=x\widetilde{S}_0^{-1}\bm\Psi(x)\widetilde{S}_0^{-T},
$$
where $\bm\Psi(x)$ is given by \eqref{weid} and 
$$
\widetilde{S}_0^{-1}=\begin{pmatrix}1&0\\-a/b&(1-c)/b\end{pmatrix}.
$$
In other words,
$$
\widehat{\bm\Psi}(x)=\frac{x\sqrt{(\sigma_+-x)(x-\sigma_-)}}{c\pi r_{11}(x)}\begin{pmatrix} cq_{11}(x) & -\frac{ac}{b}q_{11}(x)+\frac{1-c}{b}q_{12}(x) \\ -\frac{ac}{b}q_{11}(x)+\frac{1-c}{b}q_{12}(x)  & \frac{a^2c}{b^2}q_{11}(x)-\frac{2a(1-c)}{b^2}q_{12}(x)+\frac{(1-c)^2}{b^2}q_{22}(x) \end{pmatrix},
$$
where $q_{ij}(x)$ are given by \eqref{quss}.


\begin{thebibliography}{99}


\bibitem{Ber} Berezans'kii, Ju M., \emph{Expansions in Eigenfunctions of Selfadjoint Operators}, Translations of Mathematical Monographs \textbf{17}, American Mathematical Society, Rhode Island, 1968.





\bibitem{Clay} Clayton, A., {\em Quasi-birth-and-death processes and matrix-valued orthogonal polynomials}, SIAM J. Matrix Anal. Appl. \textbf{31} (2010), 2239--2260.

\bibitem{DIW} Dai, D., Ismail, M.E.H. and Wang, X., {\em Doubly infinite Jacobi matrices revisited: resolvent and spectral measure}, Adv. Math. \textbf{343}  (2019), 157--192.

\bibitem{DRSZ} Dette, H., Reuther, B., Studden, W. and Zygmunt, M., {\em Matrix measures and random walks with a block tridiagonal transition matrix}, SIAM J. Matrix Anal. Applic. \textbf{29} (2006), 117--142.















\bibitem{G2}\textrm{Gr\"unbaum, F.A.},
\textit{Random walks and orthogonal polynomials: some challenges}, Probability, Geometry and Integrable Systems, MSRI Publication, volumen \textbf{55}, 2007.

\bibitem{G1}\textrm{Gr\"unbaum, F.A.}, \emph{QBD processes and matrix orthogonal polynomials: some new explicit examples}, Numerical Methods for Structured Markov Chains, eds. D. Bini, B. Meini, V. Ramaswami, M.A. Remiche and P. Taylor, Dagstuhl Seminar Proceedings, 2008.




\bibitem{GdI2} Gr\"unbaum, F.A. and de la Iglesia, M.D.,
\textit{Matrix-valued orthogonal polynomials arising from group representation theory and a family of quasi-birth-and-death processes}, SIAM J. Matrix Anal. Applic. \textbf{30} (2008),
741--761.

\bibitem{GdI3} Gr\"unbaum, F.A. and de la Iglesia, M.D.,
\textit{Stochastic LU factorizations, Darboux transformations and urn models}, J. Appl. Prob. \textbf{55} (2018), 862--886.

\bibitem{GdI4} Gr\"unbaum, F.A. and de la Iglesia, M.D.,
\textit{Stochastic Darboux transformations for quasi-birth-and-death processes and urn models}, J. Math. Anal. Appl. \textbf{478} (2019), 634--654.












\bibitem{dI1} de la Iglesia, M. D., \emph{A note on the invariant distribution of a quasi-birth-and-death process}, J. Phys. A: Math. Theor. \textbf{44} (2011) 135201 (9pp).


\bibitem{dIJ} de la Iglesia, M.D. and Juarez, C., \emph{The spectral matrices associated with the stochastic Darboux transformations of random walks on the integers}, J. Approx. Theory \textbf{258} (2020), 105458.


\bibitem{dIR} de la Iglesia, M.D. and Rom\'an, P., \emph{Some bivariate stochastic models arising from group representation theory}, Stoch. Proc. Appl. \textbf{128} (2018), 3300--3326.


\bibitem{ILMV}  Ismail, M.E.H., Letessier, J., Masson, D., and Valent, G., {\em Birth and death processes and orthogonal polynomials}, in Orthogonal Polynomials, P.~Nevai (editor),  Kluwer Acad. Publishers (1990), 229--255.

\bibitem{KMc2}  Karlin, S. and McGregor, J., {\em The differential equations of birth and death processes, and the Stieltjes moment problem}, Trans. Amer. Math. Soc., \textbf{85} (1957), 489--546.

\bibitem{KMc3}  Karlin, S. and McGregor, J., {\em The classification of birth-and-death processes}, Trans. Amer. Math. Soc., \textbf{86} (1957), 366--400.


\bibitem{KMc6}  Karlin, S. and McGregor, J., {\em Random walks}, IIlinois J. Math., \textbf{3} (1959), 66--81.




\bibitem{LaR} Latouche, G. and Ramaswami, V., {\em Introduction to Matrix Analytic Methods in Stochastic Modeling}, ASA-SIAM Series on Statistics and Applied Probability, 1999.

\bibitem{MR}  Masson, D.R. and Repka, J., {\em Spectral theory of Jacobi matrices in $\ell^2(\mathbb{Z})$ and the $su$(1,1) Lie algebra}, SIAM J. Math. Anal., \textbf{22} (1991), 1131--1146.

\bibitem{MS} Matveev, V.B.  and Salle, M.A., \emph{Differential-difference evolution equations II: Darboux transformation for the Toda lattice}, Lett. Math. Phys. \textbf{3} (1979), 425--429.


\bibitem{Neu} Neuts, M.F., {\em Structured Stochastic Matrices of $M/G/1$ Type and Their Applications}, Marcel Dekker, New York, 1989.


\bibitem{Pru} Pruitt, W.E., \emph{Bilateral birth and death processes}, Trans. Amer. Math. Soc. \textbf{107} (1962), 508--525.








\end{thebibliography}
\end{document}